\documentclass[11pt]{amsart}
\usepackage{a4wide}
\usepackage{amssymb,amsthm}
\usepackage{fullpage}
\usepackage{color}
\usepackage{amsmath}

\newcommand{\beq}{\begin{equation}}
\newcommand{\eeq}{\end{equation}}

\newcommand{\n}{{\bf N}}

           \newcommand\om{\omega}

\def\eps{\varepsilon}

\Large

\newtheorem{prop}{Proposition}
\newtheorem{theoreme}[prop]{Theorem}
\newtheorem{lem}[prop]{Lemma}

\newtheorem{rem}[prop]{Remark}

\numberwithin{equation}{section}
\numberwithin{prop}{section}
\begin{document}
\title{ Uniform regularity for the Navier-Stokes equation  \\ with Navier boundary condition}

\author{Nader Masmoudi}
\address{Courant Institute, 251 Mercer Street, 
New York, NY 10012-1185, 
USA,}
\email{}
\author{Frederic Rousset}
\address{IRMAR, Universit\'e de Rennes 1, campus de Beaulieu,  35042  Rennes cedex, France}
\email{ frederic.rousset@univ-rennes1.fr  }
\date{}
\maketitle
\begin{abstract}
We prove that there exists an interval of time  which is  uniform in the vanishing viscosity limit
  and for which the Navier-Stokes equation with Navier boundary condition
   has a strong solution. This solution is uniformly bounded 
    in a conormal Sobolev space and has  only one normal
    derivative bounded in $L^\infty$. 
     This allows to  get  the vanishing viscosity limit  to the incompressible Euler system 
      from  a strong compactness argument.
\end{abstract} 

\section{Introduction}
We consider the incompressible Navier-Stokes equation
\beq
\label{NS}
\partial_{t} u + u \cdot \nabla u + \nabla p = \eps \Delta u, \quad  \nabla \cdot u = 0, 
 \quad x \in \Omega,
 \eeq
 in a domain $\Omega$ of $\mathbb{R}^3$. The velocity $u$ is a three-dimensional vector field
  on $\Omega$ and  the pressure $p$ of the fluid  is a scalar   function. We  add on the boundary 
 the Navier  (slip) boundary condition
 \beq
 \label{N}
  u \cdot n= 0, \quad (Su \cdot n)_{\tau}= - \alpha u_{\tau}, \quad x \in \partial \Omega
  \eeq
  where $n$ stands for the outward unit  normal to $\Omega,$ $S$ is the strain tensor,
  $$ Su= {1\over 2}( \nabla u + \nabla u^t)$$
and for some vector field $v$  on $\partial \Omega$, $v_{\tau}$ stands for the tangential part of $v$:
$ v_{\tau}= v - (v\cdot n)n$.

The parameter $\eps >0$ is  the inverse of the  Reynolds number  whereas  $\alpha $
 is another   coefficient which measures the tendancy of the fluid to slip on the boundary.
  This  type of boundary condition  is  often used  to model rough boundaries, we refer for
   example to \cite{DGV}, \cite{GM10} (see also \cite{MS03cpam} for a derivation from the 
   Maxwell boundary condition of the Boltzmann equation through a hydrodynamic limit).  
     It is known that  when $\eps $ tends to zero a weak solution
    of \eqref{NS}, \eqref{N} converges towards a solution of the Euler equation, we refer to 
  \cite{Bardos},     \cite{Mikelic-Robert}, \cite{Kelliher}, \cite{Iftimie-Planas}. In particular,  in the three-dimensional case,   in \cite{Iftimie-Planas}, 
      it is proven by a modulated energy type approach that  for sufficiently smooth solution of the Euler equation, 
        an $L^2 $ convergence holds. The situation for this problem is thus very different from
         the case of  no-slip  boundary conditions which is widely open for the Navier-Stokes
          equation except in the analytic case \cite{Sammartino-Caflisch} (see also 
           \cite{Kato86} for some necessary condition to get convergence and   
            \cite{Masmoudi98arma} for some special case).
             
            Here,  we are interested in  the existence of strong solutions  of  \eqref{NS}, \eqref{N} with uniform bounds  
  on an interval of time independent of $\eps \in (0, 1]$ and in a topology sufficiently strong
   to deduce  by a strong compactness argument that the solution converges strongly to 
    a solution of the Euler equation
    \beq
    \label{Euler}
    \partial_{t} u + u \cdot \nabla u + \nabla p = 0, \quad \nabla \cdot u = 0
    \eeq
    with the boundary condition $u\cdot n=  0$ on $\partial \Omega$.
   Note that  for  such an argument  to  succeed, we need to work in a functional space
   where both \eqref{NS} and \eqref{Euler} are well-posed.
   
   Let us recall that there are two classical ways to study  the vanishing viscosity limit  by compactness
    arguments. The first one consists 
   in trying to pass to the limit weakly in the Leray solution of the Navier-Stokes system.
   However, there is a lack of compactness and one cannot pass to the limit in the 
   nonlinear term. It is indeed an open problem to characterize the weak limit of 
   any sequence of the Navier-Stokes system when the viscosity goes to zero even in 
   the whole space case (see \cite{Lions96b,Masmoudi07hand}). The second way consists in trying to work with strong solutions in 
   Sobolev spaces. In the case of the 
   whole space (or the case there is no boundary)
     this approach yields a uniform time of existence  and the convergence towards 
   a solution of the Euler system (see \cite{Swann71,Kato72,Masmoudi07cmp}).
   The problem is that due to the presence of a boundary 
    the time of existence  $T_\eps$ depends on the viscosity and one  often 
   cannot prove that it stays bounded away from zero.  Nevertheless, 
   in domain with  boundaries,  for some special type of  Navier boundary conditions  or boundaries,
    some uniform  $H^3$ (or $W^{2,p}, $ with $p$ large enough) estimates  and   
   a   uniform time of existence for Navier-Stokes when the viscosity goes to zero 
   have been  recently   obtained  (see  \cite{Xin,BC09, Beirao, Berselli}).   As we shall see below, 
    for these special   boundary conditions,  the main part of the boundary layer vanishes which 
   allows this uniform control in some limited regularity Sobolev space. 
  
       Here, our  approach can be seen as intermediate  
   between these two cases since we  shall  get strong solutions but controlling   many  tangential derivatives and only 
    one normal  derivative. This control is compatible with the presence of a
   boundary layer when the viscosity goes to zero.

     To understand, the difficulties in the presence of  boundaries, one can use  formal boundary layer expansions.
      The solution $u^\eps$  of  \eqref{NS}, \eqref{N} is expected to have the following expansion
     \beq
     \label{expIS} u^\eps(t, x)= u(t,x) + \sqrt{\eps}\, V(t, y, z/\sqrt{\eps}) + \mathcal{O}(\eps)
     \eeq
     (we assume   that $ (y,z) \in \Omega= \mathbb{R}^2 \times(0, +\infty)$ to simplify this heuristic part)
     where $V$ is a smooth profile which is fastly decreasing in its last variable.
     Note that the rigorous constuction  of such expansions have been performed in \cite{Iftimie-Sueur}
       where  it  was also proven that the remainder is indeed  $\mathcal{O}(\eps)$ in $L^2$.
      With such an expansion, we immediately get that in the simplest space
       where the 3-D Euler equation is well-posed, namely   $H^s$, $s>5/2$ the  norm of $u^\eps$ 
       cannot be uniformly bounded because of the profile $V$.  For some special
        Navier boundary conditions  considered in   \cite{Xin,BC09, Beirao, Berselli}, the leading profile $V$
         vanishes and hence  uniform  $H^3$ or  $W^{2,p}$, $p>3$ estimates have been obtained.
         Nevertheless, as pointed out in \cite{Iftimie-Sueur}, in the generic case,  $V$ does not vanish. 
         
          We shall prove in
          this paper that in the general case, we can indeed achieve the above program by
           working in anisotropic  conormal Sobolev spaces.  Again,   because of \eqref{expIS},
            we can hope a uniform control  of one normal derivative of the solution  in $L^\infty$
             and thus  a control of the Lipschitz norm of the solution hence it seems  reasonable
              to be able to  recover in the limit the well-posedness of the Euler equation.
               The situation is thus also  different from the case of  "non-characteristic" Dirichlet condition where
                boundary layers are of size $\eps$ but  of  amplitude $1$. In this situation, one can prove in  
                some  stable cases the  $L^2$ convergence, but  since  strong  compactness in the normal variable
                 cannot  be expected, the  proof uses in a crucial way the construction of an asymptotic
                  expansion and the control of the remainder. We refer  for example to \cite{Temam-Wang},  \cite{Grenier-Gues}, \cite{Gisclon-Serre},  \cite{Rousset1D}, \cite{Grenier-Masmoudi,Masmoudi00cpam}, \cite{Grenier-Rousset}, 
                   \cite{RoussetEkman}, \cite{Masmoudi-Rousset}, \cite{Metivier-Zumbrun}.  The drawbacks of
                    this approach are  that it requires the a priori knowledge of the well-posedness of the limit problem
                     and that   it requires  the solution of the limit problem to be smoother than the one of the viscous problem (which is not very natural). Finally, let us mention that for some  problems where  only the normal
                      viscosity vanishes,  it is  also possible to use weak compactness arguments, \cite{Saint-Raymond}.
                      
                      Our aim here is to prove that  in a situation where the formal expansion
                  is under the form \eqref{expIS}, one can  get strong solutions of the viscous and
                   the inviscid problem in the same  appropriate functional framework  and justify
                    the vanishing viscosity limit by a strong compactness argument.
                      In some sense,   we want to use on  a boundary layer problem   the same approach 
                      that    is  classically used in  singular oscillatory limits (as the compressible-incompressible limit, see
                      \cite{Klainerman-Majda}, \cite{Metivier-Schochet} for example)
                      where the existence of  a strong solution on an interval of time independent of the small
                       parameter is first proven and the convergence studied in a second step.
                       To  go further in the analogy, we can think of  boundary layer problems  with formal expansions
                        as \eqref{expIS} as analogous  to  well-prepared problems. 
                              
       We  consider a  domain $\Omega \subset \mathbb{R}^3$ such   that
   there exists a covering  of $\Omega$ under the form
    $\Omega \subset \Omega_{0} \cup_{i=1}^n \Omega_{i}$ where 
    $\overline{\Omega_{0}} \subset \Omega$ and in each $\Omega_{i}$, there exists a  function
     $\psi_{i}$ such that  $\Omega \cap \Omega_i= \{ (x= (x_{1}, x_{2}, x_{3}), \, x_{3}>\psi_{i}(x_{1}, x_{2}) \}    \cap \Omega_i $
      and $\partial \Omega \cap \Omega_{i}= \{ x_{3}= \psi_{i}(x_{1}, x_{2}) \}   \cap \Omega_i.$  We say
       that $\Omega$ is $\mathcal{C}^m$ if the functions  $\psi_{i}$ are $\mathcal{C}^m$.
      
    To define Sobolev  conormal spaces,  we consider $(Z_{k})_{ 1 \leq k \leq N}$ a finite set of generators 
     of vector fields that are  tangent to $\partial \Omega$ and  we set 
     $$ H^m_{co}(\Omega)= \big\{ f \in L^2(\Omega), \quad  Z^I \in L^2 (\Omega), \quad | I | \leq m \big\}$$
      where for $I=(k_{1},  \cdots, k_{m})$, 
      $$ Z^I=  Z_{k_{1}}\cdots Z_{k_{m}}.$$
      We also set
      $$ \|f \|_{m}^2 = \sum_{|I| \leq m } \|Z^I f \|_{L^2}^2.$$
       For a vector field, $u$,  we shall say that $u$ is in $H^m_{co}(\Omega)$ if each of its  components
       are  in  $H^m_{co}$ and thus
       $$ \|u \|_{m}^2 = \sum_{i=1}^3 \sum_{| I | \leq m } \|Z^I u_{i} \|_{L^2}^2.$$
       
          In  the same way, we  set
      $$ \|u\|_{k, \infty}=  \sum_{|I| \leq m } \|Z^I u \|_{L^\infty}$$
       	and we say that $u \in W^{k, \infty}_{co}$ if  $\|u \|_{k, \infty}$ is finite.
	
	
	  Throughout the paper, we shall denote by $\| \cdot \|_{W^{k, \infty}}$ the usual Sobolev norm
	   and use the notations $\| \cdot \|$ and $(\cdot, \cdot)$ for the $L^2$ norms and scalar products.

        Note that  the $\| \cdot  \|_{m} $ norm  yields inside $\Omega$   a control
         of   the standard $H^m $ norm, whereas  close to the boundary, 
         there is no control of the normal derivatives. The use of conormal Sobolev spaces has a long history
          in (hyperbolic)  boundary value problems, we refer for example to \cite{Hormander},  \cite{Tartakoff}, \cite{Bardos-Rauch}, 
           \cite{Gues,JM09} and references therein.

      Let us set $ E^m = \{ u \in H^m_{co}, \, \nabla u  \in H^{m-1}_{co}\}$.  
       Our main result is the following:
      
      \begin{theoreme}
      \label{mainb}
      Let $m$ be an integer satisfying  $m  > 6$ and $\Omega$ be   a $\mathcal{C}^{m+2}$ domain.
      Consider  $u_{0} \in E^m$ such that  $\nabla u_{0} \in W^{1, \infty}_{co}$
       and $\nabla \cdot u_{0}= 0$, $u_{0} \cdot n_{/\partial \Omega}= 0$. Then,  there exists $T>0$
       such that for every $\eps \in (0, 1)$ and $\alpha, \, |\alpha| \leq 1$, there exists
        a unique $u^\eps  \in \mathcal{C}([0, T], E^m)$  such that $ \|\nabla u^\eps\|_{1, \infty}$ is
         bounded on $[0, T]$
         solution of \eqref{NS}, \eqref{N} with initial data  $u_{0}$.
         Moreover,  there exists $C>0$ independent of $\eps$ and $\alpha$ such that 
       \beq
       \label{uniftheo}
       \sup_{[0, T]} \Big(  \|u(t) \|_{m}+  \| \nabla u(t) \|_{m-1} + \| \nabla u(t) \|_{1, \infty} \Big) +
        \eps \int_{0}^T \| \nabla^2 u(s) \|_{m-1}^2 \, ds \leq C.
       \eeq
      
      \end{theoreme}

      Note that the uniqueness part is obvious since we work with functions with Lipschitz regularity.
       The fact that we need to  control $ \|\nabla u \|_{1, \infty}$ and not only the Lipshitz norm
        is classical in characteristic hyperbolic problems when one tries to work with the minimal normal
         regularity, we refer for example to \cite{Gues}.
          The same remark holds for  the required regularity, the same restriction on $m$ holds
           in the case of general characteristic hyperbolic problems  studied in \cite{Gues}. It is maybe
           possible to improve this by using  more precisely the structure of the incompressible equations.
            The fact that we need to control $m-1$  conormal derivatives for $\partial_{n}u $ and not only 
             $m-2$ is linked to the control of the pressure in our incompressible framework.
           The  regularity  of the domain that we require, is also mainly due to the estimate of the pressure, 
            this is the classical regularity in order to estimate the pressure in the Euler equation
             (see \cite{Temam} for example).
          Another  important remark  is that  in  proving 
         Theorem \ref{mainb}, we get
          a uniform existence time for the solution of \eqref{NS}, \eqref{N} without using  that there exists
          a solution of the Euler equation.  In particular, we shall get by passing to 
           the limit that the Euler equation is well-posed in the same functional framework.  We hope
            to be able to use this approach on more complicated problems where it is much easier
             to prove the local well-posedness for the viscous problem than for the inviscid one.
            Finally, it is also possible  to prove that in the case that the initial data is $H^s$ and satisfies some suitable
             compatibility conditions, we can deduce from the estimate \eqref{uniftheo}
              and the regularity  result for the Stokes problem   that  $u$  is  in the standard  $H^s$ Sobolev space
               on $[0, T]$.
               Nevertheless, higher order normal derivatives will not be uniformly bounded in $\eps$.

     The main steps of the  proof of Theorem \ref{mainb} are the following.
      We shall first  get a conormal energy estimate in $H^m_{co}$   for the  velocity $u$
       which is  valid  as long as the Lipshitz norm of the solution is controled.
        The second step is to estimate $\|\partial_{n} u \|_{m-1}$. In order to  get this estimate
         by   an energy method, $\partial_{n}u $ is not a convenient quantity since
          it does not vanish on the boundary.  Nevertheless, we observe
           that  $ \partial_{n} u \cdot n$ can be immediately controlled thanks to the 
           control of $u$ in  $H^m_{co}$   and  the incompressibility condition.
             Moreover, due to  the Navier condition \eqref{N},  it is convenient to study  $ \eta= ( Su \,n + \alpha u\big)_{\tau}$. Indeed it vanishes on the boundary and gives a control of $(\partial_{n} u )_{\tau}.$
              We shall thus prove a control of $\| \eta \|_{m-1}$ by performing energy estimates on the equation
               solved by $\eta$. This estimate will be valid as long as  $ \| \nabla u(t) \|_{1, \infty} $ remains
                bounded. The third step is to estimate the pressure. Indeed, since the conormal fields
                 $Z_{i}$ do not commute with the gradient,  the pressure is not transparent in the estimates.
                 We shall prove that the pressure can be split into two parts, the first one  has the same
                 regularity as in the Euler equation and the second part  is linked to  the Navier
                  condition. Finally, the  last step is to  estimate   $ \| \nabla u(t) \|_{1, \infty}$
                   and actually $\|(\partial_{n} u)_{\tau}\|_{1, \infty}$ since the other terms can be controlled 
                    by   Sobolev embedding.  To perform this estimate we shall again choose 
                     an equivalent quantity which satisfies  an homogeneous  Dirichlet condition and
                      solves at leading order a convection diffusion equation. The estimate  will
                       be obtained by  using  the fundamental solution of an approximate equation.
                       
           Once Theorem \ref{mainb} is obtained, we can easily get the  inviscid limit:
           
       \begin{theoreme}
       \label{inviscid}
       Let $m$ be an integer satisfying  $m  > 6$ and $\Omega$ be   a $\mathcal{C}^{m+2}$ domain.
      Consider  $u_{0} \in E^m$ such that  $\nabla u_{0} \in W^{1, \infty}_{co}$, $\nabla \cdot u_{0}= 0$, 
      $ u_{0} \cdot n_{/\partial \Omega}= 0$  and $u^\eps$
       the solution of \eqref{NS}, \eqref{N} with initial value $u_{0}$ given by Theorem
        \ref{mainb}. Then  there exists a unique  solution to  the Euler system \eqref{Euler},    $u \in L^\infty(0, T, E^m) $
         such that   $ \|\nabla u\|_{1, \infty}$  is
         bounded on $[0, T]$  and such that
         $$ \sup_{[0, T]}\big(  \|u^\eps - u \|_{L^2} + \|u^\eps - u \|_{L^\infty} \big) \rightarrow 0$$
         when $\eps $ tends to zero.

       \end{theoreme}
       
       We shall obtain Theorem \ref{inviscid} by a classical strong compactness argument.
       Note that  the $L^\infty$ convergence was not obtained   in   \cite{Iftimie-Planas}, \cite{Iftimie-Sueur}.
       It does not seem possible to get  such a convergence   thanks to   a modulated energy type argument.
       
       Note that if $\Omega$ is not bounded the above convergences hold   on every compact  of 
        $\overline{\Omega}$.

      The paper is organized as follows: in section \ref{sectionhalf}, we shall first  explain the main steps
       of the proof of Theorem \ref{mainb} in the simpler case where $\Omega$ is the half-space
        $\mathbb{R}^2 \times (0,  +\infty)$. This allows to present the analytical part of the proof
        without complications coming from the geometry of the domain. The general case will be treated
         in section \ref{sectionb}. Finally section \ref{sectioninviscid} is devoted to the proof of Theorem
          \ref{inviscid}.

\section{A first energy estimate}

\label{sectionL2}
In this section, we first recall the basic a priori  $L^2$ energy estimate which holds for
 \eqref{NS}, \eqref{N}.

 \begin{prop}
 \label{energie}
 Consider a (smooth) solution of \eqref{NS}, \eqref{N}, then we have for every $\eps>0$ and $\alpha\in \mathbb{R}$, 
$$
 {d \over dt } \big( {1 \over 2} \|u\|^2 \big) + 2\,  \eps \, \|  S  u \|^2 + 2 \, \alpha \eps  |u_{\tau}|^2_{L^2(\partial \Omega)}
 = 0.$$
  \end{prop} 
  
\begin{proof}
 By using \eqref{NS}, we obtain:
 $$ {d \over dt } \big( {1 \over 2} \|u\|^2 \big) = ( \eps \Delta  u, u) - (\nabla p, u) - (u \cdot  \nabla u, u)$$
  where $(\cdot, \cdot)$ stands for the $L^2 $ scalar product. Next, thanks to integration by parts
   and the boundary condition \eqref{N}, we find
\begin{eqnarray*}
& & (\nabla p, u)= \int_{\partial{\Omega} } p\, u \cdot n  -  \int_{\Omega } p\, \nabla \cdot u = 0,\\
& & (u\cdot \nabla u, u) =  \int_{\partial \Omega} {|u|^2 \over 2 } u \cdot n = 0,\\
& & ( \eps \Delta u, u) =  2 \eps (  \nabla \cdot Su, u) =  - 2  \eps \| S  u \|^2 + 2  \eps   \int_{\partial \Omega}
 \big((Su)\cdot n \big)\cdot u \, d\sigma. 
 \end{eqnarray*}
 Finally,  we get from the boundary condition \eqref{N} that
 $$ \int_{\partial \Omega}  (Su \cdot n) \cdot u  = \int_{\partial \Omega}
 \big((Su)\cdot n \big)_{\tau}\cdot u_{\tau} =  - \alpha \int_{\partial \Omega}
 |u_{\tau}|^2\, d\sigma.$$
\end{proof}

\begin{rem}
\label{Korn}
Note that  if  $\Omega$ is a Lispchitz domain,  we get from the Korn inequality
 that  for some $C_{\Omega}>0$, we have for every  
 $H^1$ vector field $u$ which is tangent to the boundary that 
 $$ \| \nabla u \|^2 \leq C_{\Omega} \big(  \| S \, u \|^2  + \|u \|^2\big).$$
  Consequently, we deduce from  Proposition \ref{energie} that
  \beq
  \label{energiebis}
 {d \over dt } \big( {1 \over 2} \|u\|^2 \big) + \eps  c_{\Omega}\|  \nabla  u \|^2 + \alpha \eps  |u_{\tau}|^2_{L^2(\partial \Omega)}
  \leq  \eps  C_{\Omega} \|u \|^2 .
  \eeq
  If $\alpha \geq 0$, this always provides a good energy estimate.
  \end{rem}
\begin{rem}
\label{bord}
Even if $\alpha \leq 0$, we get from the trace Theorem that there exists $C>0$
 independent of $\eps$ such that 
$$ |u_{\tau}|^2_{L^2 (\partial \Omega ) } \leq C  \|\nabla u\|\, \|u\| + \| u \|^2$$
 and hence, we find by using the Young inequality
 \beq
 \label{young}
  ab \leq {  \delta } a^2 +  { 1 \over 4 \delta} b^2, \quad a, \, b \geq 0, \,  \delta >0
  \eeq
   that
 \beq
 \label{ener2} {d \over dt } \big( {1 \over 2} \|u\|^2 \big) + {\eps \over 2} \, c_{\Omega}\, \| \nabla u \|^2    + \eps  |u_{\tau}|^2_{L^2(\partial \Omega)}
  \leq  {2 \,C^2 \eps  \,( \alpha^2+ 1) } \|u\|^2.\eeq
Consequently,  if $\alpha $ is  such  that $\eps \alpha^2 \leq 1$, we still  get a uniform $L^2$ 
estimate from the Gronwall Lemma.

\end{rem}
 
 \section{ The case of a half-space: $ \Omega = \mathbb{R}^3_{+}$}
 
  \label{sectionhalf}

 In order to avoid complications due to the geometry of the domain in the obtention of higher order energy estimates, we shall first give the proof of Theorem 
  \ref{mainb} in the case where  $\Omega$ is the half space $\Omega= \mathbb{R}^2 \times (0, +\infty)$.
   We shall use the notation $x=(y,z), \, z>0$ for a point $x$ in $\Omega$.  To define  the conormal
    Sobolev spaces, it suffices to use $Z_{i}= \partial_{i}$, $i=1, \, 2$ and $Z_{3}= \varphi(z) \partial_{z}$
     where $\varphi(z)$ is any smooth  bounded  function such that $\varphi(0)=0$, $\varphi'(0)\neq 0$
      and $\varphi(z)>0$ for $z>0$ (for example, $\varphi(z)=  z (1 + z )^{-1}$ fits).
       Consequently, we have
       $$ \|  u  \|_{m}^2 =\sum_{|\alpha| \leq m} \| Z^\alpha u \|^2_{L^2}, \quad
          \|  u  \|_{k, \infty}^2 =\sum_{|\alpha| \leq k} \| Z^\alpha u \|_{L^\infty}$$
                    where $ Z^\alpha = Z_{1}^{\alpha_{1}} Z_{2}^{\alpha_{2}} Z_{3}^{\alpha_{3}} u $.

 Throughout this section, we shall focus on  the proof of  a priori estimates for a sufficiently smooth
 solution of \eqref{NS}, \eqref{N} in order to get \eqref{uniftheo}. We use the symbol
  $ \lesssim $ for $\leq C$ where $C$ is a positive number which may change from line to line
   but which is independent of $\eps$ and $\alpha $ for $\eps \in (0, 1)$ and $| \alpha | \leq 1$.
   
   The aim of this section is to prove the following a priori estimate in the case of the half space
   which is  the crucial part towards the proof of Theorem \ref{mainb}.

   \begin{theoreme}
   \label{apriori+}
   For $m>6$, 
 there exists $C>0$ independent of $\eps \in (0, 1]$ and $\alpha $, $| \alpha | \leq 1$ such that   for every sufficiently smooth solution defined on $[0, T]$ of \eqref{NS}, \eqref{N}   in $\Omega=
    \mathbb{R}^2 \times (0, + \infty)$,  we have the a priori estimate
    $$ N_{m}(t) \leq C  \Big( N_{m}(0) +  (1+ t + \eps^3 t^2) \int_{0}^t \big( N_{m}(s) + N_{m}(s)^2 \big) ds\Big), \quad
     \forall t \in [0, T]$$
     where
     $$ N_{m}(t) = \|u(t) \|_{m}^2 + \| \nabla u (t) \|_{m- 1}^2 + \| \nabla u \|_{1, \infty}^2.$$ 
   
   
   \end{theoreme}

 \subsection{Conormal energy estimate}

 \begin{prop}
 
 \label{conorm+}
 For every $m\geq 0$,   a smooth  solution of \eqref{NS}, \eqref{N} satisfies the estimate
 \begin{eqnarray*}
& &  { d \over dt }  \|u(t)\|_{m}^2  +  c_{0}\,\eps \int_{0}^t \| \nabla  u \|_{m}^2
     \\
     & & \lesssim  
   \|\nabla p \|_{m-1} \|u \|_{m}   +
    \big( 1 +   \|u \|_{W^{1, \infty}} \big) \big( \|u\|_{m}^2 + \|\partial_{z} u \|_{m-1}^2 \big) 
  \end{eqnarray*}
 for some $c_{0}>0$ independent of $\eps$.
 \end{prop}
 
 
 \subsubsection*{Proof of Proposition \ref{conorm+} }
 In the proof, we shall use the notation $x=(y, z)\in \mathbb{R}^{d-1} \times (0, + \infty)$,
 $ u= (u_{h}, u_{3})\in \mathbb{R}^{d-1}\times \mathbb{R}$.
 
 The case $m=0$ just follows from Proposition \ref{energie} and Remark \ref{bord},
  and the term containing the pressure does not show up. By induction,  let us assume that
  it is proven for $k \leq m-1$.
  By applying $Z^\alpha $ to \eqref{NS} for $|\alpha | = m $, we find
 \beq
 \label{NSm}
 \partial_{t} Z^\alpha u + u \cdot \nabla Z^\alpha u + \nabla Z^\alpha p =  \eps \Delta  Z^\alpha u 
  + \mathcal{C}
  \eeq
  where the term $\mathcal{C}$ involving commutators can be written as
  $$ \mathcal{C} =  \sum_{i=1}^3 \mathcal{C}_{i}$$
  where
  \beq  
   \mathcal{C}_{1}=    -  [Z^\alpha, u \cdot \nabla ] u, \quad \mathcal{C}_{2}= -  [Z^\alpha, \nabla ]p, 
   \quad \mathcal{C}_{3}= \eps  [ Z^{\alpha}, \Delta ] u.
   \eeq
  From the divergence free condition in \eqref{NS}, we get
  \beq
  \label{dm}
  \nabla \cdot Z^\alpha u=  \mathcal{C}_d, \quad \mathcal{C}_d= -[ Z^\alpha, \nabla \cdot]  u.\eeq
  Finally, let us notice that  from the boundary condition \eqref{N} which reads explicitely in the case of a half-space
  \beq
  \label{N+}
  u_{3}= 0, \quad  \partial_{z} u_{h}=  2 \alpha u_{h}, \quad x \in \partial \Omega,
  \eeq
  we get
  \beq
  \label{N+m}
  Z^\alpha u_{3}= 0, \quad \partial_{z}Z^\alpha u_{h}= 2 \alpha Z^\alpha u_{h}+ \mathcal{C}_{b}, 
   \quad  \mathcal{C}_{b} = - \big([\partial_{z}, Z^\alpha] u_{h}\big)_{/\partial\Omega}, \quad x \in \partial \Omega.
   \eeq
   As in the proof of Proposition \ref{energie} and Remark \ref{bord},  we get from the standard energy estimate
    and the boundary condition \eqref{N+m}:
   \begin{eqnarray}
   \label{enerm1}
  & & {d \over dt } \big( {1 \over 2} \|Z^\alpha u\|^2 \big) + \eps \| \nabla Z^\alpha u \|^2
    + \eps   |Z^\alpha u_{h}|^2_{L^2(\partial \Omega)}\\
 \nonumber   & &  \lesssim  \big| \big(\mathcal{C}, Z^\alpha u \big)\big|   +
     \big|  \big( Z^\alpha p, \mathcal{C}_{d} \big) | + \eps  |\mathcal{C}_{b}|_{L^2(\partial \Omega)} \, | Z^\alpha u_{h}|
     _{L^2(\partial \Omega)} + \|u \|_{m}^2.
     \end{eqnarray}
   Indeed, since $\nabla \cdot u = 0$,  $u_{3}=0$,  and  $ Z^\alpha u_{3}=0$ on $\partial \Omega$, we  get 
    that 
  $$ \big( u \cdot \nabla Z^\alpha u , Z^\alpha u \big) = 0, \quad
   \big( \nabla Z^\alpha p, Z^\alpha u\big)= -\big(Z^\alpha p , \mathcal{C}_{d}\big).$$
   Note that when $\partial \Omega$ is not flat,    the boundary condition \eqref{N} does not imply that
    $Z^\alpha u \cdot n =0$ on $\partial \Omega$ thus a boundary term  shows up in 
     the integration by parts   and hence the estimate for the term involving the
      pressure will be worse (see the next section). 
      
   To estimate the last term in the right-hand side of \eqref{enerm1}, we can use as in Remark \ref{bord}
    the trace theorem and the Young inequality to get
   $$ \eps  |\mathcal{C}_{b}|_{L^2(\partial \Omega)} \, | Z^\alpha u_{h}|_{L^2(\partial \Omega)} \leq {1 \over 2 } \eps  \|\nabla Z^\alpha u \|^2 + C \|u\|_{m}^2 +
   C  \eps  | \mathcal{C}_{b}|_{L^2 (\partial \Omega)}^2$$
    and hence, we find
    \beq
   \label{enerm2}
   {d \over dt } \big( {1 \over 2} \|Z^\alpha u\|^2 \big) + { 1 \over 2 } \eps \| \nabla Z^\alpha u \|^2
    + \eps   |Z^\alpha u_{h}|_{L^2(\partial \Omega)}^2\lesssim  \| u\|_{m}^2    +   \big| \big(\mathcal{C}, Z^\alpha u \big)\big|   +
     \eps | \mathcal{C}_{b}|_{L^2 (\partial \Omega)}^2 +   \big|  \big( Z^\alpha p, \mathcal{C}_{d} \big) |.
     \eeq 
  
 To conclude, we need to estimate the commutators.
 First,  since
  $ [Z_{3}, \nabla \cdot ] u= - \varphi' \partial_{z} u_{3}=  \varphi' \nabla_{h} \cdot u_{h}$
  (thanks to the divergence free condition) and $[Z_{i}, \nabla \cdot ]=0$
   for $i=1, \dots, d-1$,  we easily  get that  for $m \geq 1$, 
 \beq
 \label{Cd}
 \| \mathcal{C}_{d} \| \lesssim \| u\|_{m}
 \eeq
 and hence, we obtain
 \beq
 \label{Cd2}
 \big|  \big( Z^\alpha p, \mathcal{C}_{d} \big) | \lesssim  \|u\|_{m}  \|\nabla p \|_{m-1}.
 \eeq
  We also get that 
  $$\big( [\partial_{z}, Z_{i}]u_{h}\big)_{/\partial\Omega} = 0, \quad
   \big( [\partial_{z}, Z_{3}]u_{h}\big)_{/\partial\Omega} =- \big(\varphi' \partial_{z} u_{h}\big)_{/\partial\Omega}$$
   since $\varphi$ vanishes on the boundary. Therefore, 
    from \eqref{N+}, we get 
   $$ \big( [\partial_{z}, Z_{3}]u_{h}\big)_{/\partial\Omega} = - 2 \alpha  
   \big(\varphi'  u_{h}\big)_{/\partial\Omega}.$$
   By using this last property and   the fact that $\varphi$ vanishes on the boundary, we find
 $$ | \mathcal{C}_{b}|_{L^2(\partial \Omega)} \lesssim  | u_{/\partial \Omega } |_{H^{m-1}(\partial \Omega)} $$
 and hence, we get from the trace theorem that
 \beq
 \label{Cb} \eps  | \mathcal{C}_{b}|_{L^2(\partial \Omega)}^2   \lesssim \eps \| \partial_{z} u \|_{m-1} \, \|u \|_{H^{m-1}(\partial \Omega)}.
 \eeq
 It remains to estimate $\mathcal{C}$.  First, we observe that
  \beq
 \label{C2}
 \|\mathcal{C}_{2}\| \lesssim \| \nabla p \|_{m-1}.
 \eeq
 Next,  since  we have 
 $$ [Z_{i}, \Delta ]= 0, \quad [Z_{3}, \Delta ] u= - 2 \varphi' \, \partial_{zz} u - \varphi'' \partial_{z}u, $$
 we also get by using repeatidly this property that 
 \beq
 \nonumber
 \big| \big( \mathcal{C}_{3}, Z^\alpha u \big) \big|  \lesssim     \tilde{ \mathcal{C} }_{3}  + 
    \eps \|\partial_{z} u \|_{m-1} \|u\|_{m} +
  \|u\|_{m}^2
  \eeq
 where $\tilde{\mathcal{C}}_{3}$ is given by 
 $$  \tilde{\mathcal{C}}_{3} = \sum_{ \beta, \, 0 \leq  |\beta| \leq m-1 } \  \eps \big| \big( c_{\beta }  \partial_{zz}  Z_{3}^\beta u, 
 Z^\alpha u \big) \big| $$
 for some harmless functions  $c_{\beta}$ depending on derivatives of $\varphi$. To estimate
  $ \tilde{\mathcal{C}}_{3}$, we use integration by parts. If $\beta\neq 0$,  $ |\beta| \neq 1$, since $\varphi$ vanishes
   on the boundary, we immediately  get that 
   $$ \eps \big| \big( c_{\beta }  \partial_{zz}  Z_{3}^\beta u, 
 Z^\alpha u \big) \big| \lesssim  \eps \big(  \| \partial_{z} u \|_{m} + \|u \|_{m}\big) \, \|\partial_{z} u \|_{m-1 } 
  .$$
   For $\beta = 0$ or $|\beta| = 1$, there is an additional term on the boundary,  we have    $$ 
 \big| \big( c_{\beta }  \partial_{zz}   Z^\beta  u, 
 Z^\alpha u \big) \big| \lesssim  \eps \big(  \| \partial_{z} u \|_{m} + \|u \|_{m}\big) \, \| \partial_{z} u \|_{m-1}
  + \eps | \partial_{z}u_{h} |_{L^2(\partial \Omega)}\, | Z^\alpha u|_{L^2(\partial \Omega)}.$$
  From  the boundary condition \eqref{N+} and the trace theorem, we also find 
  $$ \eps | \partial_{z}u_{h} |_{L^2(\partial \Omega)}\, | Z^\alpha u |_{L^2(\partial \Omega)} \lesssim \eps | u  |_{L^2(\partial \Omega)} \, |Z^\alpha  u |_{L^2(\partial \Omega)} \lesssim 
    \eps \|\partial_{z} u \|_{m} \|u\|_{m}.$$
   We  have  consequently proven that
  \beq
  \label{C3}
  \big| \big( \mathcal{C}_{3}, Z^\alpha u \big) \big|  \lesssim  \eps \|\partial_{z} u \|_{m}\, \big( \| \partial_{z} u \|_{m-1}
   + \|u \|_{m} \big) + \|u\|_{m}^2 + \| \partial_{z} u \|_{m-1}^2.
   \eeq
   It remains to estimate $\mathcal{C}_{1}$. By an expansion, we find 
    that $\mathcal{C}_{1}$ is under the form 
       \beq
       \label{C1+} \mathcal{C}_{1} =  \sum_{ \beta + \gamma = \alpha, \, \beta \neq 0}
    c_{\beta, \gamma} Z^\beta u \cdot Z^\gamma \nabla u  +  u \cdot  [Z^\alpha , \nabla  ]u
    .\eeq
    To estimate the last term, we  first observe that 
     \beq
     \label{C1+2}  
          \|u \cdot  [Z^\alpha , \nabla  ]u \|   \lesssim \sum_{| \beta |  \leq m-1} \| u_{3} \partial_{z} Z^\beta  u \|
          \eeq
     and then  that  because of the first boundary condition in \eqref{N+} we have
      $$ | u_{3}(t,x)| \leq \varphi(z) \|u_{3}\|_{W^{1, \infty}}.$$
      This yields
    \beq
    \label{C1+3}   \|u \cdot  [Z^\alpha , \nabla  ]u \| \lesssim  \|u_{3}\|_{W^{1, \infty}} \,\|u\|_{m}.\eeq
    To estimate the other terms, we can use the following generalized  Sobolev-Gagliardo-Nirenberg
     inequality, we refer for example to \cite{Gues} for the proof:
     \begin{lem}
     \label{gag}
       For $u, \,v\in L^ \infty \cap H^k_{co}$ , we have
     \beq
     \label{gues}
      \| Z^{\alpha_{1}} u\, Z^{\alpha_{2}} v \| \lesssim  \|u\|_{L^\infty} \, \|v\|_{k} + \|v\|_{L^\infty} \|u\|_{k}, \quad
       |\alpha_{1}| + |\alpha_{2}|=k.
      \eeq
      \end{lem}
    For $\beta \neq 0$, this immediately yields
   \begin{eqnarray}
   \label{C1+4}
    \| c_{\beta, \gamma} Z^\beta u \cdot Z^\gamma \nabla u \|
  &  \lesssim &   \|Z^\beta u_{h} \cdot Z^\gamma \nabla_{h} u \| +   \|  Z^\beta u_{3} \cdot Z^\gamma \partial_{z} u \|
  \\
  \nonumber & \lesssim  &  \|Zu \|_{L^\infty} \, \|u\|_{m} +  \|Zu\|_{L ^\infty} \,  \|\partial_{z} u \|_{m-1} + 
     \|\partial_{z} u \|_{L^\infty} \|Zu_{3}\|_{m-1} \\
  \nonumber   & \lesssim &  \|\nabla u \|_{L^\infty}\big( \|u\|_{m} + \|\partial_{z} u \|_{m-1} \big)
     \end{eqnarray}
 and hence, we find the estimate
 \beq
 \label{C1}
 \| \mathcal{C}_{1}\| \lesssim      \|\nabla u \|_{L^\infty}\big( \|u\|_{m} + \|\partial_{z} u \|_{m-1} \big).
 \eeq
From \eqref{enerm2} and  \eqref{Cd2},  \eqref{Cb}, \eqref{C2}, \eqref{C3}, \eqref{C1}
 and the remark \eqref{bord}, 
 we find
\begin{eqnarray*}
& &  {d \over dt } \big( {1 \over 2} \|u\|_{m}^2 \big) + { 1 \over 2 } \eps \| \nabla  u \|_{m}^2
      \lesssim
 \eps \|\partial_{z} u \|_{m}\big( \|\partial_{z} u \|_{m-1} + \|u \|_{m}\big)  \\
 & &+ \|u\|_{m}^2  
  + \|\nabla p \|_{m-1} \|u  \|_{m}  + \big( 1 +    
     \| u \|_{W^{1, \infty}}\big)\big( \|u\|_{m}^2  + \|\partial_{z} u \|_{m-1}^2  \big).
  \end{eqnarray*}
 To get the result, it suffices to   use  the Young inequality to absorb the term
  $\eps \| \partial_{z} u \|_{m}$ in the  left hand side.
This ends the proof of Proposition \ref{conorm+}.

 \subsection{Normal derivative estimates}
In this section, we  shall provide an estimate for $\| \partial_{z} u \|_{m-1}$. 

 A first useful remark is that  because of the divergence free condition we have
 \beq
 \label{dzu3}
  \|\partial_{z} u_{3} \|_{m-1} \leq \|u\|_{m}.
  \eeq
  Consequently, it suffices to estimate $\partial_{z}u_{h}.$
   Let us introduce the vorticity 
   $$\omega= \mbox{curl } u=\left( \begin{array}{lll} \partial_{2} u_{3}- \partial_{3} u_{2} \\
    \partial_{3}u_{1}- \partial_{1} u_{3} \\ \partial_{1} u_{2} - \partial_{2} u_{1} \end{array} \right) $$ which solves
  \beq
  \label{curl}
  \partial_{t} \omega + u \cdot \nabla \omega  - \omega \cdot \nabla u = \eps \Delta \omega, \quad x \in \Omega.
  \eeq
  On the boundary, we find thanks to \eqref{N+} that
  $$ \omega_{h} = 2 \alpha\, u_{h}^\perp, \quad x \in \partial \Omega.$$
   This leads us to introduce the  unknown
   $$ \eta=  \omega_{h} - 2 \alpha\, u_{h}^\perp.$$
   Indeed, the main advantages of this quantity is that on the boundary, we have
   \beq
   \label{etab}
   \eta= 0, \quad x \in \partial \Omega
   \eeq
   and that  we have the estimate  
   \beq
   \label{etau}
   \|\partial_{z} u_{h} \|_{m-1} \lesssim \|u\|_{m} + \|\eta \|_{m-1}.
   \eeq
   Consequently,  we shall estimate in this section $ \|\eta \|_{m-1}$.
    We have the following result:
  \begin{prop}
  \label{norm+}
  For every $m \geq 1, $ every smooth solution of \eqref{NS}, \eqref{N}, satisfies  the following estimate : 
  \begin{eqnarray*}
   &&   { d \over dt } \|\eta(t) \|^2_{m-1} + c_{0}\,\eps  \|\nabla \eta \|_{m-1}^2  \\
  \nonumber     & &  \lesssim      \|\nabla p \|_{m-1} \| \eta \|_{m-1} +
  \big(  1 +  \|u \|_{2, \infty}   +  \|\partial_{z}  u \|_{1, \infty} \big) \big( \|\eta \|_{m-1}^2 + \|u \|_{m}^2\big)
  \end{eqnarray*}
  \end{prop}
  
  \subsubsection*{Proof of  Proposition \ref{norm+} }
  From the definition of $\eta$, we find that it solves the equation
  \beq
  \label{eqeta}
  \partial_{t} \eta + u \cdot \nabla \eta - \eps \Delta \eta= \omega \cdot \nabla u_{h} + 2 \alpha \nabla_{h}^
  {\perp} p
  \eeq
  with the boundary condition \eqref{etab}.  By a standard  $L^2$ energy estimate, we find
  $$ { d \over dt }{ 1 \over 2 } \|\eta(t) \|^2 + \eps \|\nabla \eta \|^2 \lesssim 
   \Big(  \|\nabla p \| \, \|\eta \| +   \|\omega \cdot \nabla u_{h}\| \|\eta \| \Big).$$
   Furthermore,  by using that
   $$  \|\omega \cdot \nabla u_{h}\| \lesssim  \|\nabla u \|_{L^\infty}\, \|\om \|   \lesssim   \|\nabla u \|_{L^\infty}
    \big( \|\eta \| + \|u \|_{1} \big),$$
    we find the  result for $m=1$.

     Now, let us assume that  Proposition \eqref{norm+} is proven for $k \leq m-2$.
      We shall now  estimate  $ \|\eta \|_{m-1}$.
      By applying $Z^\alpha$  for $|\alpha |=m-1$ to \eqref{eqeta}, we find
      \beq
      \label{etameq}
       \partial_{t} Z^\alpha \eta + u \cdot \nabla Z^\alpha \eta  - \eps \Delta Z^\alpha \eta=
       Z^\alpha \big( \omega \cdot \nabla u_{h} \big)  + 2 \alpha Z^\alpha \nabla_{h}^\perp p+
        \mathcal{C}\eeq
        where $\mathcal{C}$ is the commutator:
        $$ \mathcal{C}= \mathcal{C}_{1} + \mathcal{C}_{2}, \quad \mathcal{C}_{1}=   [Z^\alpha , u \cdot \nabla ] \eta, 
         \quad \mathcal{C}_{2}=   - \eps [Z^\alpha, \Delta
         ] \eta.$$
         Since $Z^\alpha \eta$ vanishes on the boundary, the standard $L^2$ energy estimate for
       \eqref{etameq} yields
     \begin{eqnarray}
     \label{m1}
   &&  {d \over dt }{1 \over 2}  \|\eta(t) \|^2_{m-1} + \eps \|\nabla \eta \|_{m-1}^2  \\
     \nonumber     & &  \lesssim 
    \|\nabla p \|_{m-1} \| \eta \|_{m-1} +   \|\omega \cdot \nabla u_{h}\|_{m-1} \|\eta \|_{m-1}
    + |\big( \mathcal{C}, Z^\alpha \eta \big) \big| \Big).\end{eqnarray}
    To estimate the terms in the right-hand side, we first  write thanks to Lemma \ref{gag} that 
    \begin{eqnarray}
   \nonumber
       \|\omega \cdot \nabla u_{h}\|_{m-1} & \lesssim &   \|\omega \|_{L^\infty}\big(
     \|u_{h}\|_{m} + \|\partial_{z} u_{h}\|_{m-1}\big)+  \| \nabla  u_{h} \|_{L^\infty} \|\omega \|_{m-1} \\
    \label{n1}  & \lesssim & \|\nabla u \|_{L^\infty} \big( \|u_{h}\|_{m} +  \|\eta \|_{m-1}\big).  
     \end{eqnarray}
     Note that we have again used \eqref{etau} to get the last line.
     
     Next, we need to estimate the commutator $\mathcal{C}$.  As for  \eqref{C3}, we first
      get from integration by parts since $Z^\alpha \eta$ vanishes on the boundary  that
    \beq
    \label{nC2}
     |\big( \mathcal{C}_{2}, Z^\alpha \eta \big)\big|
      \lesssim  \eps \|\partial_{z} \eta \|_{m-1} \big( \|\partial_{z} \eta \|_{m-2} + \|\eta \|_{m-1}\big) + 
       \|\eta \|_{m-1}^2.
     \eeq
     It remains to estimate $\mathcal{C}_{1}$ which is the most difficult term.
  We can again write    
 $$ \mathcal{C}_{1} =  \sum_{ \beta + \gamma = \alpha, \, \beta \neq 0}
    c_{\beta, \gamma} Z^\beta u \cdot Z^\gamma \nabla  \eta  +  u \cdot  [Z^\alpha , \nabla  ]\eta
    .$$
    To estimate the last term, we  first observe that 
    $$   \|u \cdot  [Z^\alpha , \nabla  ]\eta \| \lesssim \sum_{k \leq m-2} \| u_{3} \partial_{z} Z_{3}^k \eta \|$$
     and  by using again that 
      \beq
      \label{u3b} | u_{3}(t,x)| \leq \varphi(z) \|u_{3}\|_{W^{1, \infty}}, \eeq
     we find 
    $$   \|u \cdot  [Z^\alpha , \nabla  ]\eta \| \lesssim  \|u_{3}\|_{W^{1, \infty}} \,\|\eta\|_{m-1}.$$ 
   To estimate  the other terms in the commutator, 
    we   write
   $$ \| c_{\beta, \gamma} Z^\beta u \cdot Z^\gamma \nabla  \eta \|
     \lesssim \| Z^\beta u_{h} \cdot Z^\gamma \nabla_{h}  \eta \|
      + \|Z^\beta u_{3} Z^\gamma  \partial_{z} \eta \|.$$
      Thanks to Lemma \ref{gag}, we have since $\beta \neq 0$ that 
  $$  \| Z^\beta u_{h} \cdot Z^\gamma \nabla_{h}  \eta \|
   \lesssim \|\nabla u \|_{L^\infty} \|\eta \|_{m-1} + \|\eta \|_{L^\infty} 
    \| Z u \|_{m-2 } \lesssim  \|\nabla u \|_{L^\infty} \big(  \|\eta \|_{m-1} +  \|u\|_{m} \big).$$
   The  remaning term is the most involved.  We want to get an estimate for which
    $\partial_{z} \eta$ does not appear.
    Indeed,  due to the expected behaviour in the boundary layer \eqref{expIS}, one cannot hope 
     an estimate which is uniform in $\eps$ for $\|\partial_{z} \eta \|_{L^\infty}$ or $\| \partial_{z} \eta \|_{m}.$
  We first write
  $$  Z^\beta u_{3} Z^\gamma  \partial_{z} \eta = { 1 \over \varphi(z)}  Z^\beta u_{3}
  \, \varphi(z)Z^\gamma  \partial_{z} \eta$$
   and then we can expand this term as a sum of terms under the form
    $$ c_{\tilde \beta, \tilde  \gamma} Z^{\tilde{\beta}} \big({ 1 \over   \varphi(z) } u_{3}\big) \, Z^{\tilde{\gamma}}\big(   \varphi
     \partial_{z} \eta\big)$$
    where $\tilde{\beta} + \tilde{\gamma} \leq m-1$,  $|\tilde{\gamma}| \neq m-1$ 
     and $c_{\tilde \beta,  \tilde \gamma}$ is some smooth bounded coefficient.
     
    Indeed,  we first notice that $Z^\alpha \varphi$ has the same properties than $\varphi$, thus
     the commutator  $[\varphi, Z^\gamma]$ can be expanded  under the form $ \tilde{\varphi}_{\tilde{\gamma}} Z^{\tilde{\gamma}}$
      with $| \tilde{\gamma }|< | \gamma |$
      where $\tilde{\varphi}_{\tilde{\gamma}}$ have the same properties as $\varphi$. 
      Then, we can  write
       $$ \tilde{\varphi}_{\tilde{\gamma}} Z^{\tilde{\gamma}} =  { \tilde{\varphi}_{\tilde{\gamma}} \over \varphi } 
        \Big( Z^{\tilde{\gamma}} 
        \big(  \varphi  \cdot\big) + [ \varphi, Z^{\tilde{\gamma}}]\Big)$$
        where the coefficient  $ \tilde{\varphi}_{\tilde{\gamma}} /\varphi$ is smooth and bounded.
         Finally, we  reiterate the process  to express the commutators
       $  [\varphi, Z^{\tilde{\gamma}}]$. Hence, after a finite number of steps, we indeed
        get that  $[\varphi, Z^\gamma]$ can be expanded as a sum of terms under the form
      $  c_{\tilde \gamma} Z^{\tilde{\gamma}} 
        \big(  \varphi \cdot\big) $ where $c_{\tilde \gamma}$ is smooth and bounded.
      In a similar way, we note that  $Z^\alpha (1/\varphi)$ has the same properties as $ 1 / \varphi$
       and hence, by the same  argument, we get that
        the commutator $[1/\varphi, Z^\beta ]$ can be expanded as a sum of terms under the form
         $ c_{\tilde \beta} Z^{\tilde{\beta}} \big( {1 \over  \varphi  } \cdot\big) $.

     If $\tilde{\beta}= 0,$ and hence $|\tilde{\gamma} | \leq m-2$,  we have
    $$ \big\| Z^{\tilde{\beta}} \big({ 1 \over \varphi(z) } u_{3}\big) \, Z^{\tilde{\gamma}} Z_{3} \eta\big\|
     \lesssim \big\|{ 1 \over \varphi(z) } u_{3} \big\|_{L^\infty} \, \|  \eta \|_{m-1}.$$
     Moreover, since $u_{3}$ vanishes on the boundary, we have
     $$ \big\|{ 1 \over \varphi(z) } u_{3} \big\|_{L^\infty}  \lesssim  \|u \|_{W^{1, \infty}}.$$
    We have thus proven that for $\tilde{\beta}= 0$
     $$ \big\| Z^{\tilde{\beta}} \big({ 1 \over \varphi(z) } u_{3}\big) \, Z^{\tilde{\gamma}} Z_{3} \eta\big\|
     \lesssim \|u \|_{W^{1, \infty}}  \|  \eta \|_{m-1}.$$
     Next, for $\tilde{\beta} \neq 0$, we can use Lemma \ref{gag} to get
     $$\big\| Z^{\tilde{\beta}} \big({ 1 \over \varphi(z) } u_{3} \big) \, Z^{\tilde{\gamma}} Z_{3} \eta\big\|
     \lesssim  \big\| Z \big({ 1 \over \varphi(z) } u_{3} \big) \big\|_{L^\infty} \, \|Z_{3} \eta  \|_{m-2}
      +   \big\| Z \big({ 1 \over \varphi(z) } u_{3} \big) \big\|_{m-2}  \|Z \eta \|_{L^\infty}.$$
      And hence, since $Z^\alpha u_{3}$ vanishes on the boundary, we get  from the Hardy inequality  that
      $$  \big\| Z \big({ 1 \over \varphi(z) } u_{3} \big) \big\|_{m-2} \lesssim \|\partial_{z}u_{3} \|_{m-1}.$$
      Indeed,  For $i=1, 2$, we directly get that 
     $$   \big\| Z_{i} \big({ 1 \over \varphi(z) } u_{3} \big) \big\|_{m-2}
      =  \big\|{ 1 \over \varphi(z) }  Z_{i}u_{3} \big\|_{m-2}
      \lesssim \|\partial_{z}u_{3} \|_{m-1}.$$
      For $i=3$,  since $Z_{3}({1 \over \varphi})$ have the same properties as $1/\varphi$, we have
      $$ \big\| Z_{3} \big({ 1 \over \varphi(z) } u_{3} \big) \big\|_{m-2}
       \lesssim   \big\|{ 1 \over \varphi(z) }  Z_{3}u_{3} \big\|_{m-2} +    \big\|{ 1 \over \varphi(z) }  u_{3} \big\|_{m-2} $$
        and hence the Hardy inequality yields
     $$   \big\| Z_{3} \big({ 1 \over \varphi(z) } u_{3} \big) \big\|_{m-2}
 \lesssim  \| \partial_{z} Z_{3} u_{3} \|_{m-2} + \|\partial_{z} u_{3} \|_{m-2} \lesssim  \|\partial_{z} u_{3}\|_{m-1}.$$
  By using again the divergence free condition, we thus  get that
$$  \big\| Z \big({ 1 \over \varphi(z) } u_{3} \big) \big\|_{m-2} \lesssim \|\partial_{z}u_{3} \|_{m-1} \lesssim  \|u\|_{m}$$

 Consequently, we obtain that 
     $$ \big\| Z^{\tilde{\beta}} \big({ 1 \over \varphi(z) } u_{3} \big) \, Z^{\tilde{\gamma}} Z_{3} \eta\big\|
      \lesssim \big(  \|u \|_{2, \infty} + \|Z\eta \|_{L^\infty} \big) \big( \|\eta \|_{m-1} + \|u \|_{m}\big).$$
      We have thus proven that
      \beq
      \label{C1n} \| \mathcal{C}_{1} \| \lesssim 
      \big(  \|u \|_{2, \infty} + \|u\|_{W^{1, \infty}}  + \|Z\eta \|_{L^\infty} \big) \big( \|\eta \|_{m-1} + \|u \|_{m}\big).
     \eeq 
     To end the proof of Proposition \ref{norm+}, it suffices to collect \eqref{m1}, \eqref{n1}  \eqref{nC2} and \eqref{C1n}.

     \subsection{Pressure estimates}
     
     It remains to estimate the pressure
      and the $L^\infty $norms in the right hand side of the estimates of Propositions  \ref{norm+}
       and \ref{conorm+}
 
 The aim of this section is  to  give  the estimate  of $ \|\nabla p\|_{m-1} $.
 
 \begin{prop}
 \label{pressure}
 For every $m \geq 2$, there exists $C>0$ such that for every $\eps \in (0, 1]$,  a smooth solution
  of \eqref{NS}, \eqref{N} on $[0, T]$ satisfies the estimate
 $$ \| \nabla p(t) \|_{m-1} \leq C \Big(  \eps \| \nabla u(t) \|_{m-1} +  ( 1 + \|u(t) \|_{W^{1, \infty}} ) \big( \|u(t) \|_{m}+
  \| \partial_{z} u (t) \|_{m-1} \Big), \quad \forall t \in [0, T].$$
  \end{prop}
 Note that  by  combining  Proposition \ref{pressure},   Proposition \ref{conorm+},  Proposition \ref{norm+}
  and \eqref{dzu3}, \eqref{etau},  we find that 
  \begin{eqnarray}
  \label{dzconorm++} & &
   \|u(t) \|_{m}^2  +   \|\partial_{z}  u(t) \|_{m-1}^2   + 
    \eps \int_{0}^t\big( \| \nabla u \|_{m}^2 + \| \nabla^2 u \|_{m- 1}^2\big)
     \\\nonumber   & &
     \lesssim  \|u_{0} \|_{m}^2 +   \|\partial_{z}  u_{0} \|_{m-1}^2  
 + \int_{0}^t  
  \big( 1 +   \|u \|_{2, \infty}   +  \|\partial_{z}  u \|_{1, \infty} \big) \big( \|\partial_{z} u  \|_{m-1} ^2+ \|u \|_{m}^2\big).
  \end{eqnarray}
  In particular, we see from this estimate that  it only remains to control 
   $    \|u \|_{2, \infty}   +  \|\partial_{z}  u \|_{1, \infty}$.

 The proof of Proposition \ref{pressure} relies on the following estimate for
  the Stokes problem in a half-space.
  Consider the system
  \beq
  \label{stokessys}
  \partial_{t } u   - \eps \Delta u + \nabla p = F, \quad \nabla \cdot u = 0, \quad z>0, 
   \eeq
    with the Navier boundary condition \eqref{N} which reads
   \beq
   \label{navierS}u_{3}= 0, \quad \partial_{z} u_{h}= 2 \alpha u_{h}, \quad z=0\eeq
   where $F$  is  some given source term.
   
   We have the following estimates for the Stokes problem
   \begin{theoreme}
 \label{theoStokes}
For every $m \geq 2$, there exists $C>0$ such that for every $t \geq 0$, we have the estimate
$$ \ \| \nabla p \|_{m-1} \leq C   \Big( \| F\|_{m-1} +  \| \nabla \cdot F \|_{m-2} +  \eps  \| \nabla u \|_{m-1} + \| u \|_{m-1}
  \Big).$$

 \end{theoreme}
 The proof can be obtained  from standard elliptic regularity results.  Nevertheless, 
  in  the case of  a half-space, the proof follows easily from explicit computations in   the Fourier side.
    We shall thus sketch the proof for the sake of completeness.
 \subsubsection*{Proof of Theorem \ref{theoStokes}}
 By taking the divergence of \eqref{stokessys}, we get that $p$ solves
  $$  \Delta p =  \nabla \cdot F, \quad z>0.$$
  Note that   in this proof, the time will be only a parameter, for notational convenience,
    we shall not
    write down explicitely that all the involved functions  depend on it.
     
  From the third component of the velocity equation, we get that
  \beq
  \label{dzp0} \partial_{z}p(y,0) =   \eps \partial_{zz} u_{3}(y,0) +  \eps \Delta_{h} u_{3}(y,0) -
   \partial_{t}u_{3}(y,0)+F_{3}(y,0).\eeq
 From,  the boundary condition for the velocity,  we have  that
  $$  \Delta_{h} u_{3}(y,0) = 0, \quad \partial_{t} u_{3}(y,0)=0.$$  Moreover by applying $\partial_{z}$ to the divergence free condition, 
  we get that
  $$ \partial_{zz} u_{3}(y,0)=  -  \nabla_{h} \cdot  \partial_{z}u_{h}(y,0)$$
   and hence  from the  second boundary condition in \eqref{navierS}, we obtain
   $$  \partial_{zz} u_{3}(y,0)= -  2 \alpha \nabla_h \cdot u_{h}.$$
    Consequently,  we can use \eqref{dzp0} to express the pressure on the boundary  and we obtain the following elliptic equation with Neumann  boundary condition for
     the pressure:
     \beq
     \label{p}
       \Delta p =  \nabla \cdot F, \quad z>0, \quad  \partial_{z}p(y,0)=  -  2 \alpha \eps \nabla_{h}\cdot u_{h}(y,0)
         + F_{3}(y,0).
        \eeq
   Note that we can express $p$ as $ p= p_{1}+ p_{2}$ where $p_{1}$ solves
   \beq
   \label{p1+}
     \Delta p_{1}=  \nabla \cdot F, \quad z>0, \quad  \partial_{z}p_{1}(y,0)= F_{3}(y,0)
    \eeq
    and $p_{2}$ solves
    \beq
   \label{p2+}
     \Delta p_{2}= 0, \quad z>0, \quad  \partial_{z}p_{2}(y,0)=  2 \alpha \eps \nabla_{h}\cdot u_{h}(y,0)
       . 
    \eeq
    The meaning of this decomposition is that  $p_{1}$  corresponds
     to the gradient part of the  usual Leray-Hodge decomposition of the vector field $F$
      whereas $p_{2}$ is purely determined by the Navier  boundary condition. The  desired estimates
       for $p_{1}$ and $p_{2}$ can be obtained from standard elliptic theory. In the case of
        our very simple geometry, the proof is very easy thanks to the explicit representation of the solutions
         in Fourier space.

  To estimate $p_{1}$, we can use an explicit representation of the solution in Fourier space
  (we refer for example to the appendix of \cite{Masmoudi-Rousset}). 
 By   taking  the Fourier transform in the $(x_{1}, x_{2})$ variable, we get  that $\hat{p}_{1}$ solves
 \beq
 \label{eqp1} \partial_{{zz}} \hat{p}_{1} - |\xi|^2 \hat{p}_{1} = i \xi \cdot \hat{F}_{h} + \partial_{z} \hat{F}_{3}, \quad
  z>0, \quad \partial_{z}\hat{p}_{1}(\xi, 0)= \hat{F}_{3}(\xi,0).\eeq
  Consequently the resolution of this ordinary differential equation gives
  $$  \hat{p}_{1}(\xi,z)= \int_{0}^{+ \infty} G_{\xi}(z,z') \hat{F}(\xi,z')\, dz'$$
  where    $G_{\xi}(z,z')$ is defined as
  \begin{eqnarray*}
  G_{\xi}(z,z')= &&   -\Big(  e^{- |\xi| z'} { \mbox{cosh}( |\xi| z ) \over |\xi|} i\xi, \,  
   e^{- |\xi| z'} \mbox{cosh }(|\xi|z) \Big), \quad  z<z', \\
    & &    -\Big(  e^{- |\xi| z} { \mbox{cosh}( |\xi| z' ) \over |\xi|} i\xi, \,  
   -  e^{- |\xi| z} \mbox{sinh }(|\xi|z') \Big), \quad  z>z'.
   \end{eqnarray*}
   Note that the product  $G_{\xi} \hat{F}$ has to be understood as the product of a $(1, 3)$ matrix
    and a $(3,1)$ matrix.
    
   In  particular, we obtain  that
    $$  \partial_{z}\hat{p}_{1}(\xi,z)= \int_{0}^{+ \infty}  K_{\xi}(z,z') \hat{F}(\xi,z')\, dz' +\hat{F}_{3}(\xi,z)$$
    where $K_{\xi}(z,z')$ is defined by
    \begin{eqnarray*}
     K_{\xi}(z,z') = & & \partial_{z} G_{\xi}(z,z'), \quad z<z' \\
      & & \partial_{z} G_{\xi}(z,z') , \quad z>z'.
    \end{eqnarray*}
     Since
     $$ \sup_{z,\, \xi} \big( \|K_{\xi}(z, \cdot) \|_{L^1(0, + \infty)} + |\xi| \|G_{\xi}(z, \cdot)\|_{L^1(0,+ \infty)}  \big)<+\infty$$
      and
      $$ 
       \sup_{z', \,\xi}\big( \|K_{\xi}(\cdot, z') \|_{L^1(0,+\infty)} + |\xi | \|G_{\xi}(\cdot, z')\|_{L^1(0,+ \infty)}  \big) <+ \infty,$$
    we get by   
       using the Schur Lemma  that
    $$\| \partial_{z}  \hat{p}_{1} (\xi, \cdot )\|_{L^2(0, +\infty)}
     + |\xi| \| \hat{p}_{1}(\xi, \cdot ) \|_{L^2(0, + \infty)}
      \leq C  \|\hat{F}(\xi , \cdot) \|_{L^2(0, +\infty)},$$
      where $C$ does not depend on $\xi$.
       Hence,   by using the Bessel identity we obtain from the previous estimate  that 
      $$  \|\nabla p_{1} \|_{L^2} \lesssim \| F \|_{L^2}.$$
    
    In a similar way, we get by multiplication in the Fourier side that
    $$ \| \nabla_{h}^{k} p_{1} \| \lesssim \| F \|_{k}, \quad \forall k \leq m-1.$$
     Moreover, by using \eqref{eqp1}, we also obtain that
     $$ \| \partial_{zz} p_{1} \|_{k} \lesssim  \| \nabla \cdot  F \|_{k}, \quad \forall k \leq m-2.$$
      Consequently,  since $[ \partial_{zz}, Z_{3}] = \varphi''  \partial_{z}+ 2 \varphi' \partial_{zz}$,
       the    result for  $p_{1}$ follows easily  by   applying $Z_{3}^{\alpha_{3}}$ to \eqref{eqp1}
        and by induction on $\alpha_{3}$. 
       This yields finally  
    \beq
    \label{nablap1} \| \nabla p_{1} \|_{m-1}
    \lesssim    \| F  \|_{m-1} + \| \nabla \cdot F \|_{m-2}.\eeq  
  which is the desired estimate for $p_{1}$.
  
  Let us turn to the estimate of $p_{2}$. Again, by using the Fourier transform, we can solve explicitely
   \eqref{p2+}. We obtain that
   \beq
   \label{eqp2} \hat{p}_{2}(\xi, z) = 2 i \alpha \eps  {\xi\over |\xi| } \cdot \hat{u}_{h}(\xi,0)  e^{- |\xi  | z } 
    . \eeq
     From the Bessel identity, this yields
   $$ \| \nabla p_{2} \|_{m-1} \lesssim \eps \, | \alpha|\,  \| u_{h}(\cdot, 0)\|_{H^{ m - {1 \over 2}}(\mathbb{R}^2)}
    $$
     and hence  from the Trace Theorem, we obtain
  \beq
  \label{nablap2}
   \| \nabla p_{2} \|_{m-1} \lesssim \eps \, | \alpha|\, \|  \nabla u_{h}\|_{m-1}^{1 \over 2 }
    \| u_{h}\|_{m-1}^{1 \over 2 }.
    \eeq 
Consequently, we can  collect \eqref{nablap1}, \eqref{nablap2}  to get the result.
 This ends the proof of Theorem \ref{theoStokes}.
 
%
 It remains the:
 \subsubsection*{Proof of   Proposition \ref{pressure} }
  We can first use  Theorem \ref{theoStokes} with     $ F= - u \cdot \nabla u $  
    to get
 $$   \| \nabla p \|_{m-1}  \lesssim   \|u \cdot \nabla u \|_{m-1} + \| \nabla u \cdot \nabla u \|_{m-2}+   \eps  \| \nabla u \|_{m-1}
 + \|u \|_{m-1} .$$
 Since, by using again Lemma \ref{gag}, we have
 \begin{eqnarray*}
 & &  \|u \cdot \nabla u \|_{m-1} \lesssim  \|u \|_{W^{1, \infty}} \big( \|u \|_{m-1} + \| \nabla u \|_{m-1}
  \big) \lesssim   \|u \|_{W^{1, \infty}} \big( \|u \|_{m} + \| \partial_{z} u \|_{m-1}\big),  \\
   & &  \| \nabla u \cdot \nabla u \|_{m-2} \lesssim  \| \nabla u \|_{L^\infty} \, \| \nabla u \|_{m-2},
   \end{eqnarray*}
   the proof of Proposition \ref{pressure} follows.

 \subsection{$L^\infty$ estimates}
 In this section, we shall provide the $L^\infty$ estimates which are needed  to estimate the right-hand sides
  in the estimates of  Propositions \ref{conorm+}, \ref{norm+}.
   Let us set 
   \beq
   \label{Nt}
   Q_{m}(t) = \|u(t) \|_{m}^2+ \|\eta(t) \|_{m-1}^2 + \|\eta \|_{1, \infty}^2
   \eeq 
 \begin{prop}
 \label{Linfty}
 For $m_{0}>1$, we have
 \begin{eqnarray}
 \label{un+}
& &  \|u \|_{W^{1, \infty}} \lesssim   \|u \|_{m_{0}+ 2 } +  \|\eta \|_{m_{0}+ 1 }+ \|\eta \|_{L^\infty} \leq
 Q_{m}^{1 \over 2}(t), \quad m \geq m_{0}+ 2 \\
\label{deux+}& &  \|u \|_{2, \infty} \lesssim \|u \|_{m_{0}+ 3 } + \|\eta \|_{m_{0}+ 2 } 
\leq Q_{m}^{1 \over 2}(t), \quad m\geq m_{0}+ 3, \\
 \label{deux++}& &  \| \nabla u \|_{1, \infty} \lesssim  \|u \|_{m_{0}+ 3} + \| \eta \|_{m_{0}+ 3} +  \| \eta \|_{1, \infty} \lesssim Q_{m}^{1 \over 2 }(t), \quad m \geq m_{0}+ 3 
  \end{eqnarray}
   
 \end{prop}
 
 From  this proposition and \eqref{dzconorm++}, we see that we shall  only need
   to estimate  $ \| \eta \|_{1, \infty}$ in order to conclude.

 \begin{proof}
  We easily get \eqref{un+}, \eqref{deux+} and \eqref{deux++} from  the anisotropic Sobolev embedding :
    \beq
    \label{sob}
     \|f \|_{L^\infty}^2 \lesssim 
     \big\| |f|_{H^{m_{0}}(\partial \Omega)}\big\|_{L^\infty_{z}}^2 \lesssim \|\partial_{z}f \|_{m_{0}} \, \|f\|_{m_{0}} + 
      \|f\|_{ m_{0}}^2,
      \eeq
      where we use the notation
     $$ \big\| |f|_{H^{m_{0}}(\partial \Omega)}\big\|_{L^\infty_{z}} = \sup_{z} |f(\cdot, z)|_{H^{m_{0}}(\mathbb{R}^2)}, $$
      the divergence free condition which provides
  $$ |\partial_{z} u_{3}(t,x)| \leq |\nabla_{h} u_{h}(t,x) |$$
   and the fact that by definition of $\eta$, we have 
   $$ |\partial_{z} u_{h}(t, x) | \lesssim |\nabla_{h} u_{3}(t,x)| + |u_{h }(t,x)| + |\eta (t,x) |.$$
 
 \end{proof}
   
  We shall next  estimate $\|\eta \|_{L^\infty} $ and $\| Z \eta \|_{L^\infty}.$
  Note that  we cannot estimate these two quantities by using  \eqref{sob}.
   Indeed, we do not expect  $\partial_{z} \eta \sim \partial_{zz}u$ to be uniformly bounded in conormal spaces in the boundary
    layer (recall that $u$  is expected to behave as$ \sqrt{\eps}\, U(z/\sqrt{\eps}, y)$   as shown in 
     \cite{Iftimie-Sueur}). Consequently, we  need to use more carefully the properties of the equation
      for $\eta$ to get these needed  $L^\infty$ estimates directly. This is the aim of the following proposition.

 \begin{prop}
 \label{etainfty}
We have,  for $m >6 $,  the  estimate:
$$    
     \|  \eta(t)  \|_{1, \infty }^2  \lesssim
       Q(0) 
                +(1+ t+ \eps^3 t^2)   \int_{0}^t  \big(    Q_{m}(s)^2  +  Q_{m}(s) \big)\, ds. $$   
 
 \end{prop}
 
 \subsubsection*{Proof of Proposition \ref{etainfty}}
  
  The estimate of $ \|\eta \|_{L^\infty}$ is a  consequence of the maximum principle for 
    the transport-diffusion equation  \eqref{eqeta}. 
     Let us set 
 \beq
 \label{Fdef}
  F=  \omega \cdot \nabla u_{h}  + 2 \alpha \nabla_{h}^\perp p \eeq 
   so that \eqref{eqeta} reads 
  \beq
  \label{eqetaF}
  \partial_{t} \eta + u \cdot \nabla \eta = \eps \Delta\eta +F .
  \eeq
 We obtain that 
   $$
  \|\eta(t) \|_{L^\infty} \leq \|\eta_{0} \|_{L^\infty} + \int_{0}^t \|F \|_{L^\infty} 
 $$
    and hence from the Cauchy-Schwarz inequality that
    \beq
    \label{eta0}
   \|\eta(t) \|_{L^\infty}^2 \leq \|\eta_{0} \|_{L^\infty}^2 + t  \int_{0}^t \|F \|_{L^\infty}^2 
    \eeq    
   Next, we want to get a similar estimate for $Z_{i} \eta$.  The main difficulty is   the estimate
    of $Z_{3} \eta $ since  the commutator of  this vector field  with
     the Laplacian involves two derivatives in the   normal variable.

      Let $\chi(z)$ be a smooth compactly supported function  which  takes the value  one in the vicinity of $0$
       and is supported in $[0, 1]$.
        We can write
       $$\eta = \chi \eta + (1 - \chi ) \eta := \eta^b + \eta^{int}$$ 
       where $\eta^{int}$ is supported away from the boundary and $\eta^b$ is compactly
        supported in $z$.
        
        Since $ 1 - \chi $ and   $\partial_{z} \chi$ vanish in the vicinity of the boundary, and that  our conormal $H^m$ norm is equivalent to the usual $H^s$ norm away from the
         boundary, we can write thanks to the usual  Sobolev embedding that
        $$ \| \eta^{int} \|_{1, \infty} \lesssim \|\kappa u \|_{H^{s_{0}}}, \quad s_{0}> 2+ {3 \over 2}$$
         for some  $\kappa$  supported away from the boundary and hence we get that
         \beq
         \label{etaint}
          \| \eta^{int}(t) \|_{1, \infty} \lesssim  \|u \|_{m} \lesssim Q_{m}(t)^{1 \over 2 }, \quad
           m \geq 4.
           \eeq
           
       Consequently,    it only  remains   to estimate   $\eta^b$. We first notice that
          $\eta^b$ solves the equation
       \beq
       \label{etabeq}
          \partial_{t} \eta^{b} + u  \cdot \nabla  \eta^{b} = \eps \Delta  \eta^{b} +  \chi F +
          \mathcal{C}^{b},\eeq
         in the half-space $z>0$  with homogeneous Dirichlet boundary condition,   
          where $\mathcal{C}^{b}$ is the commutator
          $$ \mathcal{C}^{b}= - 2 \eps \partial_{z} \chi\, \partial_{z} \eta - \eps  \partial_{zz} \chi \,\eta   +  u_{3} \partial_{z} \chi  \,   $$
          Note that again since $\partial_{z} \chi$ and $\partial_{zz} \chi$ are supported away
           from the boundary, we have from the usual Sobolev embedding that
      $$   \| \mathcal{C}^b \|_{1, \infty} \lesssim  \| \kappa u \|_{W^{3, \infty}} \lesssim \| \kappa u \|_{H^{s_{0}}}, 
          \quad s_{0} >3+ {3 \over 2}$$
          and hence that
          \beq
         \label{Cb+} 
           \| \mathcal{C}^b \|_{1, \infty} \lesssim  \| u \|_{m} \leq Q_{m}^{1 \over 2},  \quad m \geq  5.
           \eeq
     A crucial estimate towards the proof of Proposition \eqref{etainfty} is the following:      
        \begin{lem}
        \label{lemFP0}
        Consider  $\rho$  a smooth solution of 
        \beq
        \label{eqetaFP0} 
        \partial_{t} \rho + u \cdot \nabla \rho = \eps \partial_{zz} \rho + \mathcal{S}, \quad  z>0, \quad 
         \rho(t,y,0)= 0
         \eeq
          for  some smooth divergence free vector field $u$ such that 
          $u \cdot n= u_{3}$ vanishes on the boundary. Assume that 
         $ \rho$ and $\mathcal{S}$  are  compactly supported in $z$.  Then, we have the estimate:
         $$ \| \rho (t) \|_{1, \infty}
          \lesssim \| \rho_{0} \|_{1, \infty} + 
          \int_{0}^t \Big( \big(  \|u \|_{2, \infty} + \| \partial_{z} u \|_{1, \infty} \big) \big( \| \rho \|_{1, \infty}
           + \| \rho \|_{m_{0}+ 3} \big) + \| \mathcal{S} \|_{1, \infty} \Big)$$
            for $m_{0}>2$.

        \end{lem}
        Let us  first explain how we can use the result of Lemma \ref{lemFP0} to conclude.  
        By  applying Lemma \ref{lemFP0} to \eqref{etabeq} with $\mathcal{S}= 
         \chi F + \mathcal{C}^b + \eps \Delta_{y } \eta^b$ (where $\Delta_{y}$ is the Laplacian
          acting only on the $y$ variable),  we immediately get that
       \begin{eqnarray}
       \label{etabpf1} & &  \| \eta^b (t) \|_{1, \infty}
          \lesssim \| \eta_{0} \|_{1, \infty}  \\
           & & \nonumber
         +   \int_{0}^t \Big( \big(  \|u \|_{2, \infty} + \| \nabla u \|_{1, \infty} \big) \big( \| \eta \|_{1, \infty}
           + \| \eta \|_{m_{0}+ 3} \big) + \| \mathcal{C}^{b} \|_{1, \infty} + \| F \|_{1, \infty} + \eps  \| \Delta_{y}
            \eta^b \|_{1, \infty} \Big)\end{eqnarray}
           Note that $ \|\mathcal{C}^b \|_{1, \infty}$ is well controlled thanks to \eqref{Cb+}
            and
             that thanks to Lemma \ref{Linfty}, we have
                          \beq
    \label{FFP}  \|  F \|_{1, \infty} \lesssim 
         \| \nabla_{h} p \|_{1, \infty} +  \| \omega \|_{1, \infty}  \|\nabla u \|_{1, \infty} 
          \lesssim   \| \nabla_{h} p \|_{1, \infty} + Q_{m}.\eeq
       From the anisotropic Sobolev embedding \eqref{sob}, we note that
       $$   \| \nabla_{h} p \|_{1, \infty}\lesssim \| \nabla p \|_{m-1}$$
       for $m-1 \geq m_{0}+ 2 \geq 5$.  Finally, we also notice  that thanks to a new use of  \eqref{sob}, we have that
      \begin{eqnarray*}
       \Big( \eps  \int_{0}^t \| \Delta_{y} \eta^b \|_{1, \infty} \Big)^2 
      &  \lesssim &  \eps^2 \Big( \int_{0}^t \| \nabla^2 u \|_{m-1}^{1 \over 2  } Q_{m}^{1 \over 4} \Big)^2 + \eps^2 t \int_{0}^t Q_{m} \\
     &  \lesssim  &  \eps^2 t  \Big(\int_{0}^t \| \nabla^2 u \|_{m-1}^2 \Big)^{1 \over 2}
        \Big(\int_{0}^t Q_{m} \Big)^{1 \over 2} + 
      \eps^2 t \int_{0}^t Q_{m} \\
       &  \lesssim  & 
           \eps \int_{0}^t \| \nabla^2  u \|^2_{m- 1} +  (\eps^2 t + \eps^3 t^2) \int_{0}^t Q_{m}
       \end{eqnarray*}
            for $m \geq m_{0}+ 4$.
          Consequently, we get from \eqref{etabpf1} and \eqref{etaint} that
      \begin{eqnarray*}
       \| \eta (t) \|_{1, \infty}^2 &  \lesssim&     \| \eta_{0} \|_{1, \infty}
        + Q_{m}(t) + \eps \int_{0}^t \| \nabla^2  u \|_{m-1}^2
         + t \int_{0}^t \big(Q_{m}(s)^2 + \|\nabla p(s) \|_{m-1}^2\big)\, ds\\
          & & \quad   + (1 + t + \eps^3 t^2 ) \int_{0}^t Q_{m}\, ds.
          \end{eqnarray*}
       Finally, we get  from this  last estimate
         and   \eqref{dzconorm++} that
       $$   \| \eta (t) \|_{1, \infty}^2 \lesssim Q(0) + (1+ t+ \eps^3 t^2)  \int_{0}^t\big( Q_{m}(s) + Q_m(s)^2\big) \,ds.$$
        This ends the proof of Proposition \ref{etainfty}.

It remains to prove Lemma \ref{lemFP0}
\subsubsection*{Proof  of Lemma \ref{lemFP0}}
         
          The estimate of $ \| \rho \|_{L^\infty}$ and  $\|\partial_{i} \rho\|_{L^\infty}= \|Z_{i} \rho\|_{L^\infty} $, $i=1, \, 2$  also  follow
           easily from the maximum principle.
          Indeed, we get that $\partial_{i} \rho$ solves the equation
          $$  \partial_{t}  \partial_{i} \rho + u  \cdot \nabla \partial_{ i }  \rho = \eps \partial_{zz}
           \partial_{ i } \rho +   \partial_{i} \mathcal{S} 
           - \partial_{i} u\cdot  \nabla \rho$$
           still with an homogeneous Dirichlet boundary condition. Consequently, by using again the
            maximum principle, we find
          \beq
          \label{hetab} \|\nabla_{h} \rho \|_{L^\infty} \leq  \|\eta_{0}\|_{1, \infty}
           + \int_{0}^t \Big(  \|  \mathcal{S} \|_{1, \infty} 
            +  \| \partial_{i} u\cdot \nabla \rho \|_{L^\infty}\Big).\eeq
            To estimate the last term in the above expression, we write again
            \beq
            \label{hetab1}   \| \partial_{i} u\cdot \nabla \rho \|_{L^\infty}
             \lesssim \| u \|_{{1, \infty} } \| \rho\|_{1, \infty}
              +  \| \partial_{z} \partial_{i}u_{3}\|_{L^\infty}  \| Z_{3} \rho \|_{L^\infty}
               \lesssim \|u \|_{2, \infty} \| \rho \|_{1, \infty}.\eeq
               by a new use of the fact that $u$ is divergence free.

      It remains to estimate $\|Z_{3} \rho \|_{L^\infty} $  which is the most difficult term. We cannot
       use the same  method as previously due to the bad commutator between $Z_{3}$ and
        the Laplacian. We shall use  a more precise description of the solution of \eqref{etabeq}.
           We shall first  rewrite the equation \eqref{etabeq} as
        $$  \partial_{t} \rho + z \partial_{z}u_{3}(t,y,0) \partial_{z} 
         \rho + u_{h}(t,y,0) \cdot \nabla_{h} \rho - \eps \partial_{zz} \rho  =
          \mathcal{S} - R:=G$$
          where 
          $$ R=  \big(u_{h}(t,x)- u_{h}(t,y,0)\big) \cdot \nabla_{h} \rho +\big( u_{3}(t,x) -  z \partial_{z}u_{3}(t,y,0)
          \big) \partial_{z} \rho.$$
          The idea will be to use an exact representation of the Green's function of the operator
           in the left-hand side to perform the estimate.

          Let   $S(t, \tau)$ be the $C^0$  evolution operator generated by the left hand side of the above equation.
          This means that  $f(t,y,z)= S(t, \tau) f_{0}(y,z)$  solves the equation
   $$
     \partial_{t} f +   z \partial_{z}u_{3}(t,y,0)  \partial_{z} f  
         + u_{h}(t,y,0) \cdot \nabla_{h} f - \eps \partial_{zz} f= 0, \quad z>0, \, t>\tau, \quad f(t, y, 0) = 0. 
$$
with the initial condition $f(\tau, y, z) = f_{0}(y,z)$.
Then we have the following estimate:
         \begin{lem}
         \label{FP}
         There exists $C>0$  such that
        $$ \big\|  z\partial_{z}  S(t, \tau) f_{0} \|_{L^\infty}
         \leq C \big( \|f_{0}\|_{L^\infty} + \|  z \partial_{z}  f_{0} \|_{L^\infty} \big), \quad \forall t \geq \tau \geq 0.$$
           \end{lem}
           
           We shall  postpone the proof of the Lemma until the end of the section.
           
       By using Duhamel formula, we deduce that 
      \beq
      \label{Duh}
      \rho(t) = S(t, \tau) \rho_{0} + \int_{0}^t S(t, \tau) G(\tau) \, d\tau.
      \eeq
      Consequently, by using Lemma \ref{FP}, we obtain
      $$ \|  Z_{3} \rho \|_{L^\infty} \lesssim
        \Big( \|\rho_{0}\|_{L^\infty}
        + \|   z\partial_{z}  \rho_{0} \|_{L^\infty}
         + \int_{0}^t  \big( \|G \|_{L^\infty}
        + \|  z\partial_{z}  G \|_{L^\infty}\big) \Big).$$
        Since $\rho$ and $G$ are compactly supported, we obtain
      \beq
      \label{etab1} 
       \|  Z_{3} \rho \|_{L^\infty} \lesssim
        \Big( \|\rho_{0}\|_{1, \infty}
                + \int_{0}^t   \|G \|_{1, \infty} \Big).
                \eeq
        It remains to estimate the right hand side.
    First, let      us estimate the term involving $R$.
           Since $u_{3}(t,y,0)= 0$, we have
          $$ \| R \|_{L^\infty} \lesssim \|u_{h}\|_{L^\infty} \|\nabla_{h} \rho \|_{L^\infty} + \| \partial_{z} u_{3}\|_{L^\infty}
            \| Z_{3} \rho \|_{L^\infty} \lesssim  \|u\|_{1, \infty} \, \| \rho \|_{1, \infty}.$$
            Note that we have used again the divergence free condition to get the last estimate.
            Next, in a similar way, we  get
         $$ \| Z R\|_{L^\infty} \lesssim  \| u \|_{2, \infty}  \| \rho \|_{1, \infty}
          + \Big\|  \big(u_{h}(t,x)- u_{h}(t,y,0)\big) \cdot Z \nabla_{h} \rho \Big\|_{L^\infty}
           + \Big\|  \big( u_{3}(t,x) -  z \partial_{z}u_{3}(t,y,0)
          \big) Z \partial_{z} \rho\Big\|_{L^\infty}$$
          By using the Taylor formula and
            the fact that $\rho$ is compactly supported in  $z$,  this yields 
            $$ \| ZR \|_{L^\infty} \lesssim 
            \| u \|_{2, \infty}  \| \rho \|_{1, \infty} +
             \| \partial_{z}u_{h}\|_{L^\infty} \| \varphi(z)   Z \nabla_{h} \rho \|_{L^\infty}+
             \|\partial_{zz} u_{3}\|_{L^\infty} \| \varphi^2(z) Z \partial_{z} \rho \|_{L^\infty}. $$
             Consequently, by using the divergence free condition, we get
            $$ \| R \|_{L^\infty} \lesssim \big(  \|u\|_{2, \infty} + \| \partial_{z} u  \|_{1, \infty}\big) \big( \|  \rho\|_{1, \infty}
             + \| \varphi(z) \rho \|_{2, \infty}\big).$$
             The additional factor $\varphi$ in the last term is crucial to close our estimate.
              Indeed,  by  the Sobolev embedding \eqref{sob}, we  have  that  for $|\alpha |=2$
               $$ \| \varphi Z^\alpha \eta \|_{L^\infty} 
                \lesssim  \| Z^\alpha \eta \|_{m_{0}} + \| \partial_{z} \big ( \varphi Z^\alpha \eta\big)  \|_{m_{0}} $$
                 and hence we obtain by definition of $Z_{3}$ that
            \beq
            \label{trick}   \| \varphi Z^\alpha \eta \|_{L^\infty}  \lesssim  \|\eta \|_{m_{0} + 3 }, 
             \quad | \alpha | = 2.\eeq
             Consequently, we finally get by using Proposition \ref{Linfty} that for $m \geq m_{0}+ 4$
            \beq
            \label{Rest}
            \| R(t) \|_{1, \infty} \lesssim
            \big(  \|u\|_{2, \infty} + \| \partial_{z} u  \|_{1, \infty}\big) \big( \|  \rho\|_{1, \infty}
             + \|  \rho \|_{m_{0}+ 3}\big).
            \eeq 
          Finally,  the proof of      
      Proposition \ref{etainfty} follows from the last estimate and \eqref{etab1}.

     It remains to prove Lemma \ref{FP}.

    \subsubsection*{Proof of Lemma \ref{FP}}
    Let us set $f(t,y,z)= S(t, \tau) f_{0}(y,z)$, then $f$ solves the equation
   $$
     \partial_{t} f +   z \partial_{z}u_{3}(t,y,0)  \partial_{z} f  
         + u_{h}(t,y,0) \cdot \nabla_{h} f - \eps \partial_{zz} f= 0, \quad z>0,  \quad f(t, y, 0) = 0. 
$$
We can first transform the problem into a problem in the whole space.
 Let us  define $\tilde{f}$ by 
 \beq
 \label{tildef} \tilde{f}(t,y,z)= f(t,y,z), \, z>0, \quad \tilde{f}(t,y,z) = - f(t,y,-z), \, z<0\eeq
 then $\tilde{f}$ solves
 \beq
 \label{FP2}
   \partial_{t} \tilde{f} +   z \partial_{z}u_{3}(t,y,0)  \partial_{z} \tilde{f}  
         + u_{h}(t,y,0) \cdot \nabla_{h} \tilde{f} - \eps \partial_{zz} \tilde{f}= 0,  \quad  z \in \mathbb{R}  
 \eeq 
with the initial condition $\tilde{f}(\tau, y, z)= \tilde{f}_{0}(y,z)$.  

%
%
We shall get the estimate by using an exact
   representation of the solution.

To solve \eqref{FP2}, we can first define
\beq
\label{tildeg} g(t,y,z)= f(t, \Phi(t, \tau, y), z)\eeq

where $\Phi$ is the solution of 
$$
\partial_{t} \Phi = u_{h}(t,\Phi, 0), \quad \Phi(\tau, \tau, y)= y.
$$
Then, $g$ solves the equation
$$ \partial_{t}g +  z \gamma(t,y) \partial_{z}g - \eps \partial_{zz} g = 0, \quad z \in \mathbb{R}, \quad
 g(\tau, y, z)= \tilde{f}_{0}(y,z)$$
 where
 \beq
 \label{Gamma}
 \gamma(t,y)= \partial_{z} u_{3}(t, \Phi(t, \tau, y), 0)
 \eeq
 which is a one-dimensional  Fokker-Planck type equation (note that now $y$ is only a parameter
  in the problem). By a simple computation in Fourier space, we find the explicit representation
 \begin{eqnarray*} 
 g(t,x) & =  &   \int_{\mathbb{R}}
  {1 \over \sqrt{ 4 \pi \eps  \int_{\tau}^t  e^{2 \eps ( \Gamma(t) - \Gamma(s) ) }\, ds}} 
   \exp \Big(  - {  (z- z')^2 \over  4 \eps  \int_{\tau}^t  e^{2 \eps ( \Gamma(t) - \Gamma(s) ) }\, ds} \Big)
     \tilde{f}_{0}(y,  e^{- \Gamma(t) }z')\, dz' \\
     & = &  \int_{\mathbb{R}} k(t, \tau, y, z-z') \tilde{f}_{0}  (y,  e^{- \Gamma(t) }z')\, dz'
     \end{eqnarray*} 
     where $\Gamma(t)= \int_{\tau}^t \gamma(s,y)\, ds$ (note that $\Gamma$ depends
      on $y$ and $\tau$, we do not write down explicitely this dependence for notational convenience).
      
      Note that $k$ is non-negative and that $\int_{\mathbb{R}} k(t, \tau, y, z) \, dz= 1$, thus, 
      we immediately recover that
      $$ \|g \|_{L^\infty} \leq \|\tilde{f}_{0}\|_{L^\infty}.$$
      Next, we observe that   we  can write 
    $$ z \partial_{z}k (t,\tau, z-z')=\big( z - z' \big) \partial_{z} k  -  z'  \partial_{z'}k
     (t, \tau, z-z')$$
     with
     $$  \int_{\mathbb{R} } \big| \big( z - z' \big) \partial_{z} k \big| dz' \lesssim 1$$
    and thus by using an integration by parts, we find
    $$ \|z \partial_{z} g\|_{L^\infty} \lesssim 
     \| \tilde{f} \|_{L^\infty} + \Big\| e^{- \Gamma(t) } \int_{\mathbb{R}} k(t,\tau, y, z') z' \partial_{z} \tilde{f}_{0}(y, e^{- \Gamma(t)}
      z')  dz'  \Big\|_{L^\infty}.$$
      By using \eqref{Gamma}, this yields
    $$   \| z  \partial_{z} g\|_{L^\infty} \lesssim 
     \| \tilde{f}_{0} \|_{L^\infty} + \| z \partial_{z} \tilde{f}_{0} \|_{L^\infty}.$$
     By using \eqref{tildef} and \eqref{tildeg},  we obtain
     $$ \| z  \partial_{z} f \|_{L^\infty} \lesssim     \| z  \partial_{z} \tilde{f} \|_{L^\infty}
      \lesssim  \| \tilde{f}_{0} \|_{L^\infty} +   \| z \partial_{z} \tilde{f}_{0} \|_{L^\infty}
       \lesssim   \| f_{0} \|_{L^\infty} +   \| z \partial_{z} f_{0} \|_{L^\infty}  .$$
       This ends the proof of Lemma \ref{FP}. 
      
    \subsection{Final a priori estimate}
    By  combining Propositions \ref{etainfty}, \ref{Linfty} and  \eqref{dzconorm++}, the proof of Theorem \ref{apriori+}
     follows.

     
   \section{The case of a general domain with smooth boundary}
   
   \label{sectionb}
   
   \subsection{Notations and conormal spaces}
   We   recall that  $\Omega$ is a bounded domain of $  \mathbb{R}^3$ and we assume that
   there exists a covering  of $\Omega$ under the form
    \beq
    \label{covomega}
    \Omega \subset\Omega_{0} \cup_{i=1}^n \Omega_{i}
    \eeq where 
    $\overline{\Omega_{0}} \subset \Omega$ and in each $\Omega_{i}$, there exists a smooth function
     $\psi_{i}$ such that  $\Omega \cap \Omega_i= \{ (x= (x_{1}, x_{2}, x_{3}), \, x_{3}>\psi_{i}(x_{1}, x_{2}) \}  \cap \Omega_i $
      and $\partial \Omega \cap \Omega_{i}= \{ x_{3}= \psi_{i}(x_{1}, x_{2}) \}   \cap \Omega_i.$

    To define Sobolev  conormal spaces,  we consider $(Z_{k})_{ 1 \leq k \leq N}$ a finite set of generators 
     of vector fields that are  tangent to $\partial \Omega$ and  
     $$ H^m_{co}(\Omega)= \big\{ f \in L^2(\Omega), \quad  Z^I \in L^2 (\Omega), \quad | I | \leq m \big\}$$
      where for $I=(k_{1},  \cdots, k_{m})$, 
      We use the notation  
       $$ \|u \|_{m}^2 = \sum_{i=1}^3 \sum_{| I | \leq m } \|Z^I u_{i} \|_{L^2}^2$$
    and in the same way
      $$ \|u\|_{k, \infty}=  \sum_{|I| \leq m } \|Z^I u \|_{L^\infty}, $$
        $$ \| \nabla Z^m u  \|^2 = \sum  _{|I| \leq m } \| \nabla Z^I u \|_{L^2}^2.$$

        Note that, by   using our covering of $\Omega$, we can always  assume that  each   vector field is  
        supported in  one of the  $\Omega_{i}$, moreover,   in $\Omega_{0}$ the $\| \cdot  \|_{m} $ norm  yields   a control
         of   the standard $H^m $ norm, whereas  if $\Omega_{i}\cap \partial\Omega \neq \emptyset$, 
         there is no control of the normal derivatives. 
         
         In the proof $C_{k}$ will denote a number independent of $\eps \in (0, 1]$ which depends only on 
        the $\mathcal{C}^k$ regularity of  the boundary, that is to say on the $\mathcal{C}^k$ norm of the functions $\psi_{i}.$

          By using that $\partial \Omega$ is given locally by $x_{3}= \psi(x_{1}, x_{2})$ (we omit the subscript $i$ for notational
           convenience),  it is convenient to use  the coordinates:
          \beq
          \label{coord} \Psi:\,  (y,z)\mapsto (y, \psi(y)+z).\eeq
           A local basis is thus given by the vector fields
           $ (\partial_{y^1}, \partial_{y^2}, \partial_{z})$.
            On the  boundary $\partial_{y^1}$ and $\partial_{y^2}$ are tangent  to $\partial \Omega$, but
            $\partial_{z}$  is not a normal vector field. We shall sometimes use the notation $\partial_{y^3}$
             for $\partial_{z}$.
              By using this parametrization,   we can  take  as  suitable vector fields
                  compactly supported in $\Omega_{i}$ in the definition of the $\| \cdot \|_{m}$ norms:
          $$ Z_{i}=  \partial_{y^i}=  \partial_{i}+ \partial_{i}\psi\, \partial_{z}, \quad i=1, \, 2,  \quad
           Z_{3}=  \varphi(z)\big( \partial_{1} \psi \, \partial_{1}+ \partial_{2} \psi \, \partial_{2}-
           \partial_{z}\big)$$
           where  $\varphi$ is smooth, 
            supported  in $\mathbb{R}_{+}$,  and such that $\varphi(0)= 0$, $\varphi(s)>0, \, s>0$.

            In this section, we shall still denote  by $\partial_{i}$, $i=1,\, 2, \, 3$ or $\nabla$
             the derivation with respect to the standard coordinates of $\mathbb{R}^n$.
           The coordinates of a vector field $u$ in the basis $(\partial_{y^i})_{1 \leq i \leq 3}$  will be denoted by $u^i$, thus
            $$u= u^1 \partial_{y^1} + u^2 \partial_{y^2}+ u^3 \partial_{y^3}$$
            whereas we shall still denote by $u_{i}$ the coordinates in the canonical basis of $\mathbb{R}^3$, namely $ u = u_1 \partial_1 +u_2 \partial_2  + u_3 \partial_3$
            (we warn the reader that this convention does not match with the standard Einstein convention
             for raising and lowering the indices in differential geometry).
                       
           We shall also denote by ${ n}$ the  unit outward normal which is given locally by 
           $$ n(\Psi(y,z))= { 1 \over \big(1 +| \nabla \psi(y) |^2\big)^{1 \over 2} }
            \left(  \begin{array}{ll}  \partial_{1} \psi(y) \\
                   \partial_{2} \psi(y) \\ - 1  \end{array} \right)$$
                   (note that  $n$ 
                    is actually  naturally defined  in the whole  $\Omega_{i}$ and does not depend on $x_{3}$)
                   and in the same way, by  $\Pi$ the orthogonal projection
                   $$ \Pi(\Psi(y,z)) X=  X- X\cdot n(\Psi(y,z)) \, n(\Psi(y,z))$$
                   which  gives the orthogonal projection
                    onto the tangent space of the boundary. 
                    
                    By using these notations, the Navier boundary
                     condition \eqref{N} reads:
                 \beq
                 \label{N2} u \cdot n= 0, \quad  \Pi \partial_{n} u = \theta(u)  -  2 \alpha  \Pi u
                 \eeq
                     where $\theta$ is the shape operator (second fondamental form) of the boundary i.e  given by
                   $$ \theta (u)= \Pi  \big(u \cdot \nabla n\big).$$

       The crucial step in the proof of Theorem \ref{mainb} is again the proof of an a priori estimate.
       We shall prove that:
       
       \begin{theoreme}
   \label{apriorib}
   For $m>6$,  and $\Omega$ a $\mathcal{C}^{m+2}$ domain, 
 there exists $C_{m+2}>0$ independent of $\eps \in (0, 1]$ and $\alpha $, $| \alpha | \leq 1$ such that   for every sufficiently smooth solution defined on $[0, T]$ of \eqref{NS}, \eqref{N},  we have the a priori estimate
    $$ N_{m}(t) \leq C_{m+2}  \Big( N_{m}(0) +  (1+ t + \eps^3 t^2) \int_{0}^t \big( N_{m}(s) + N_{m}(s)^2 \big) ds\Big), \quad
     \forall t \in [0, T]$$
     where
     $$ N_{m}(t) = \|u(t) \|_{m}^2 + \| \nabla u (t) \|_{m- 1}^2 + \| \nabla u \|_{1, \infty}^2.$$ 
   
   \end{theoreme}

      The steps of the  proof of Theorem \ref{apriorib}  are  the same  as  in the proof of Theorem \ref{apriori+}.
      Nevertheless some new difficulties will appear mainly due to  the fact that $n$ is not a constant vector 
      field any more.     
      \subsection{Conormal energy estimates}
      
      \begin{prop}
 
 \label{conormb}
 For every $m$,   the solution of \eqref{NS}, \eqref{N} satisfies the estimate
 \begin{eqnarray*}
& &   \|u(t)\|_{m}^2  + \eps \int_{0}^t \| \nabla  u \|_{m}^2
     \\
     & & \leq C_{m+2} \Big( \|u_{0}\|_{m}^2 + \int_{0}^t
      \Big( 
   \|\nabla^2 p_{1} \|_{m-1} \, \|u \|_{m}+  \eps^{-1} \|\nabla p_{2}\|_{m-1}^2    + 
   \big( 1 +   \|u \|_{W^{1, \infty}} \big)\big( \|u\|_{m}^2 + \| \nabla u \|_{m-1}^2 \big) \Big)
  \end{eqnarray*}
 where the pressure $p$ is splitted as $p= p_{1} + p_{2}$ where $p_{1}$ is the "Euler" part of
  the pressure which solves
  $$ \Delta p_{1}=  - \nabla \cdot (u \cdot \nabla u), \quad x \in \Omega, \quad \partial_{n}p=
   - \big(u \cdot \nabla u \big) \cdot n, \quad x \in \partial \Omega$$
    and $p_{2}$ is the "Navier Stokes part"  which solves
   $$ \Delta p_{2}= 0, \quad x \in \Omega, \quad \partial_{n}p_{2}= \eps  \Delta u \cdot n, \quad x \in \partial \Omega.$$ 
 \end{prop}
 Note that   the estimate involving the pressure is worse than in Proposition
  \ref{conorm+}.    Indeed, since    $Z^\alpha u \cdot n$ does not vanish on the boundary, we cannot
   gain one derivative in the estimate of the Euler part of the pressure  by using an integration by parts.

 \subsection{Proof of Proposition \ref{conormb}}
 The estimate for $m=0$ is already given in Proposition \ref{energie}.
  Assuming  that it is proven for  $k \leq m-1$, we shall prove it for $k=m \geq 1$.
  By applying $Z^I$ for $| I |= m $  to \eqref{NS} as before, we obtain that
  \beq
  \label{NSbalpha} \partial_{t}Z^I u  + u \cdot \nabla Z^I u + Z^I \nabla p = \eps Z^I \Delta u
   + \mathcal{C}^1\eeq
   where $\mathcal{C}^1$ is  the commutator defined as
   $$ \mathcal{C}^1= \big[ Z^I, u \cdot \nabla ].$$
    By using again Lemma \ref{gag}, we obtain that
    \beq
    \label{C1b}
      \| \mathcal{C}^1 \| \leq C_{m+1}  \|  u \|_{W^{1, \infty}} \big( \|u \|_{m} + \| \partial_{z}u \|_{m-1} \big).
     \eeq
    Indeed, we can perform this estimate in each coordinate patch.  In $\Omega_{0}$, this is
    a direct consequence of  the standard tame Gagliardo-Nirenberg-Sobolev inequality.
    Close to the boundary, we first notice that $   u\cdot \nabla u= u_{1} \partial_{1} u  + u_{2} \partial_{2} u+ u_3 \partial_{3} u $  can be written  
    $$ u\cdot \nabla u= u_{1} \partial_{y^1} u  + u_{2} \partial_{y^2}u+ u \cdot \n \, \partial_{z} u$$
     where $u_{i}$, $i= 1, \, 2$  and $3$ are the coordinates of $u$ in the standard canonical basis of $\mathbb{R}^n$
      and $\n$  is  defined by
      $$ \n=  \left(\begin{array}{ccc} -  \partial_{1} \psi \\ - \partial_{2} \psi  \\1 \end{array} \right) . $$
       Note, that when the boundary is  given by $x_{3}= \psi(x_{1}, x_{2})$, $\n$ 
       is a normal (non-unitary) vector field. 
      Moreover,  we also have that $Z^I = \partial_{y^1}^{\alpha_{1}} \partial_{y^2}^{\alpha_{2}}
       (\varphi(z) \partial_{z})^{\alpha_{3}}$. Since $u\cdot \n$ vanishes on the boundary $z=0$, we
         can use the same estimates as in \eqref{C1+}, \eqref{C1+2}, \eqref{C1+3}, \eqref{C1+4}
         with $u_{3}$ replaced by $u \cdot \n$.
       Note that we have the estimate
       $$ \|u \cdot \n \|_{m} \leq C_{m+1} \|u \|_{m}$$
       which explains the dependence in $C_{m+1}$ in \eqref{C1b} and also that
       $$ \| \partial_{z} u \|_{m-1} \leq C_{m} \| \nabla u \|_{m-1}.$$
       
       Consequently, a standard energy estimate for \eqref{NSbalpha} yields
       \beq
       \label{estconormb1} {d \over dt} {1 \over 2} \|Z^I u\|^2 \leq   \eps \int_{\Omega} Z^I \Delta u \cdot Z^I u
        - \int_{\Omega } Z^I \nabla p \cdot Z^Iu + C_{m+1}  \|  u \|_{W^{1, \infty}} \big( \|u \|_{m} + \| \partial_{z}u \|_{m-1} \big) \|u\|_{m}\eeq
        We shall first estimate the first term above in the right hand side. 
        To evaluate this term through integration by parts, we  shall  need  estimates of  the trace of  $u$ on the boundary.
       At first, thanks to the Navier boundary condition under the form  \eqref{N2}, we have
      that
      \beq
      \label{Pi1b} | \Pi  \partial_{n} u |_{H^{m}(\partial\Omega)} \leq | \theta(u)  |_{H^{m}(\partial \Omega)}
       +  2 \alpha  |u|_{H^{m}(\partial \Omega)} 
        \leq  C_{m+2} |u |_{H^{m}(\partial \Omega)}, \quad \forall m\geq 0.\eeq
       To estimate the  normal part of $\partial_{n}u$, we can use that $\nabla \cdot u$= 0.
       Indeed,  we  have
       \beq
       \label{divcoord} \nabla \cdot u= \partial_{n}u \cdot n + \big( \Pi\partial_{y^1} u \big)^1 + 
        \big(\Pi \partial_{y^2} u \big)^2
        \eeq
        and hence, we immediately get that
      \beq
      \label{Pi2b}
       | \partial_{n} u \cdot n |_{H^{m-1}(\partial \Omega)} \leq C_{m} | u |_{H^m(\partial \Omega)}.
       \eeq
       Note that by combing these two last estimates, we have in particular that
       \beq
       \label{Pi3b}
       | \nabla u |_{H^{m-1}(\partial \Omega)} \leq C_{m+1}  |u|_{H^m(\partial \Omega)}.
       \eeq
       Finally, let us notice that since $u\cdot n=0$ on the boundary,  we have that
       \beq
       \label{Piun3}
       |   (Z^\alpha u) \cdot n |_{H^1(\partial \Omega)} \leq C_{m+2} | u |_{H^m(\partial \Omega)}, \quad
        | \alpha |=m.\eeq 
        
       Next, we can write that
       \begin{eqnarray*}
        \eps\int_{\Omega}   Z^I \Delta u \cdot   Z^I u
       &  =  &  2  \eps  \int_{\Omega}   \big(\nabla \cdot Z^I   S u  \big)\cdot  Z^I u
         +   \eps \int_{\Omega} \big([Z^I, \nabla \cdot]  S  u\big) \cdot  Z^Iu \\  &= &   I +  II.
         \end{eqnarray*}
         By integration by parts, we get for the first term that
         $$ I= - \eps \int_{\Omega} Z^I  S u \cdot \nabla Z^I u + \eps \int_{\partial \Omega}
          \big( (Z^I  S u)\cdot n \big)  \cdot Z^I u$$
          and we note that
     $$ - \eps \int_{\Omega} Z^I  S u \cdot \nabla Z^I u =  -\eps \|  S ( Z^I u) \|^2
      + \eps \int_{\Omega} [Z^I,  S] u \cdot \nabla Z^I u.$$
      Consequently, thanks to the  Korn inequality, there exists $c_{0}>0$ (depending only  $C_{1}$)
       such that 
     $$  \eps \int_{\Omega} Z^I  S u \cdot \nabla Z^I u  \leq   - c_{0}\eps \|   \nabla  ( Z^I u) \|^2
      + C_{1} \|u \|_{m}^2
      + \eps \int_{\Omega} [Z^I,  S] u \cdot \nabla Z^I u.$$
      Moreover,  the commutator term can be bounded by
      $$ \Big| \eps \int_{\Omega} [Z^I, S] u \cdot \nabla Z^I u \Big|
       \leq C_{m+1} \,  \eps  \| \nabla Z^m u \|\, \|  \nabla u \|_{m- 1}.$$
       It remains to estimate the boundary term in the expression for $I$. 
       We can first notice that
       $$   \int_{\partial \Omega}
          \big( (Z^I S u)\cdot n \big)  \cdot Z^I u= \int_{\partial \Omega} Z^I\big( \Pi \big( Su  \cdot  n \big) \big)
           \cdot \Pi Z^I u + \int_{\partial \Omega} Z^I\big( \partial_{n} u \cdot n \big)\, Z^I u \cdot n 
            + \mathcal{C}_{b}$$
          where the commutator term $\mathcal{C}_{b}$ can be bounded by 
      $$| \mathcal{C}_{b}| \leq C_{m+1} |\nabla u |_{H^{m-1}(\partial \Omega)}  |u|_{H^{m}(\partial \Omega)}
 \leq C_{m+1}   |u|_{H^{m}(\partial \Omega)}^2$$
  thanks to a new use of \eqref{Pi3b}. 
  For the main term, we   write that  thanks to the Navier boundary condition \eqref{N}  we have     
        $$ \Big| \int_{\partial \Omega} Z^I\big( \Pi \big( Su  \cdot  n \big) \big)
           \cdot \Pi Z^I u  \Big| \leq C_{m+1} |u |_{H^m(\partial \Omega)}^2$$
     and that by integrating once along the boundary, we have that
   $$\Big|  \int_{\partial \Omega} Z^I\big( \partial_{n} u \cdot n \big)\, Z^I u \cdot n  \Big|
    \lesssim |\partial_{n}u \cdot n |_{H^{m-1}(\partial \Omega)} \, | Z^I u \cdot n |_{H^1(\partial \Omega)}
    \leq C_{m+2} |u |_{H^m(\partial \Omega)}^2$$
    where the last estimate comes from \eqref{Pi2b}, \eqref{Piun3}.
    
             We have thus proven that
           $$  \eps \Big| \int_{\partial \Omega}
          \big( (Z^I  S u)\cdot n \big)  \cdot Z^I u \Big|
           \leq C_{m+2}\, \eps \, |u|_{H^m(\partial \Omega)}^2.$$
           This yields
        \beq
        \label{conormbI}  I \leq 
        - \eps c_{0} \| \nabla Z^I u  \|^2 + C_{m+2} \big(   \eps  \| \nabla Z^m u \| \big( \|u\|_{m} + \| \nabla u \|_{m-1}\big) 
         + |u |_{H^m(\partial \Omega)}^2 \big).\eeq
         It remains to estimate $II$.
          We can expand $[Z^I, \nabla \cdot]$ as a sum of terms under the form
          $\beta_{k}  \partial_{k} Z^{\tilde{I}}$ with $|\tilde{I}| \leq m-1$ and $|\beta_{k}|_{L^\infty} \leq C_{m+1}$.
           Consequently,  we need to estimate
           $$\int_{\Omega} \beta_{k}\partial_{k }\big(Z^{\tilde{I}} S u\big) \cdot Z^I u.$$
            By using an integration by parts, we get that
          $$ \eps  \Big| \int_{\Omega} \beta_{k}\partial_{k }\big(Z^{\tilde{I}} S u\big) \cdot Z^I u \Big|
           \leq  C_{m+2}\, \eps\,\Big(  \| \nabla Z^{m-1} u\| \, |\nabla Z^m u\|
            +  \|u\|_{m}^2 +  |\nabla u |_{H^{m-1}(\partial \Omega)} \, |u|_{H^m(\partial \Omega)}\Big).$$
            Consequently, from a new use of  \eqref{Pi3b} we get that
          \beq
          \label{conormbII}
          | II | \leq   
          C_{m+2}\, \eps\,\Big(  \| \nabla Z^{m-1} u\| \, \|\nabla Z^m u\|
            +  \|u\|_{m}^2 +  |u|_{H^m(\partial \Omega)}^2\Big).\eeq
           To estimate the  term involving  the pressure in \eqref{estconormb1}, we write 
          $$
           \Big|  \int_{\Omega} Z^I \nabla p \cdot Z^I u \Big|
           \leq  \|\nabla^2 p_{1} \|_{m-1}\, \|u\|_{m} +  \Big| \int_{\Omega} Z^I \nabla p_{2} \cdot Z^I u \Big|.$$
           For the last  term, we   have
         $$  \Big| \int_{\Omega} Z^I \nabla p_{2} \cdot Z^I u \Big| \leq  
          \Big| \int_{\Omega} \nabla Z^I p_{2} \cdot Z^I u  \Big| + C_{m+ 1 } \| \nabla p_{2} \|_{m-1} \, \|u \|_{m}$$
           and we can integrate by parts to get
        $$   \Big| \int_{\Omega} \nabla Z^I p_{2} \cdot Z^I u  \Big|
         \leq  \| \nabla p_{2} \|_{m-1} \, \| \nabla Z^I u \|
          + \Big| \int_{\partial \Omega} Z^I p_{2} \, Z^I u \cdot n \Big|.$$
          To control the boundary term, when $m \geq 2$,  we  integrate by parts once along the boundary to obtain
          $$  \Big| \int_{\partial \Omega} Z^I p_{2} \, Z^I u \cdot n \Big| \leq C_{2} \| Z^{\tilde{I}}p_{2} \|_{L^2(\partial \Omega)} \, \| Z^I u \cdot n \|_{H^1(\partial \Omega)}$$
          where $ \tilde{I}= m-1.$ Next, we use 
           \eqref{Piun3} and
            the trace Theorem to get  that 
        $$  \Big| \int_{\Omega} Z^I \nabla p_{2} \cdot Z^I u \Big| \leq  C_{m+2} 
         \| \nabla p_{2} \|_{m-1} \big( \| \nabla Z^I u \| + \|u \|_{m} \big).$$
         We have thus proven that
        $$   \Big|  \int_{\Omega} Z^I \nabla p \cdot Z^I u \Big|
           \leq C_{m+2} \Big(   \|\nabla^2 p_{1} \|_{m-1}\, \|u\|_{m}    +  \| \nabla p_{2} \|_{m-1} \big( \| \nabla Z^I u \| + \|u \|_{m} \big) \Big).$$
        
        Consequently, by collecting the previous estimates, we    
          deduce from \eqref{estconormb1} that
       \begin{eqnarray}
      \nonumber  & & {d \over dt} {1 \over 2} \| u\|_{m}^2  + \eps \, c_{0}\,  \| \nabla Z^m u \|^2 \\
    \nonumber  & &   \leq    C_{m+2 } \Big( \eps  \| \nabla Z^m u \| \, ( \|u\|_{m} + \|\nabla Z^{m-1} u \|) 
          + |u |_{H^m(\partial \Omega)}^2  \\
        \nonumber & & \quad     +    \|\nabla^2 p_{1} \|_{m-1}\, \|u\|_{m}  + \| \nabla p_{2} \|_{m-1} \big(  \| \nabla Z^m u \| + \|u \|_{m}\big)
            +\big( 1 +   \|  u \|_{W^{1, \infty}}\big) \big( \|u \|_{m}^2 + \| \partial_{z}u \|_{m-1} \big)^2  \Big).
        \end{eqnarray} 
          By using the  Trace Theorem 
         and the Young inequality, we finally get that
         \begin{eqnarray}
      \nonumber  {d \over dt} {1 \over 2} \| u\|_{m}^2  +{c_{0} \over 2 } \eps  \| \nabla Z^m u \|^2
      &   \leq  &  C_{m+2 } \Big(   \eps  \|\nabla Z^{m-1} u \|_{m}^2
        +    \|\nabla^2 p_{1} \|_{m-1} \|u \|_{m}  + \eps^{-1} \| \nabla p_{2} \|_{m-1}^2  \\
         & & \quad  \nonumber     + \big( 1+   \|  u \|_{W^{1, \infty}}\big) \big( \|u \|_{m}^2 + \| \partial_{z}u \|_{m-1} ^2\big) \Big)
        \end{eqnarray}
        and the result follows by using the induction assumption to control
        $  \eps  \|\nabla Z^{m-1} u \|_{m}^2$. 
         This ends the proof of Proposition \ref{conormb}.
         
        \subsection{Normal derivative estimates}
         In view of Proposition \ref{conormb}, we shall now provide an estimate for $\| \nabla u \|_{m-1}$.
          Of course, the only difficulty is   to estimate $ \| \chi \, \partial_{z} u \|_{m-1}$
           or  $ \| \chi \, \partial_{n} u \|_{m-1}$
           where $\chi$ is  compactly supported in one of the $\Omega_{i}$  and  with value one in a
           vicinity of the boundary.  Indeed, we have by definition of the norm that
           $ \| \chi \, \partial_{y^i} u \|_{m-1} \leq C_{m}\|u\|_{m}$, $  i=1, \, 2$.
            We shall thus use the local coordinates \eqref{coord}.

          At first, thanks to \eqref{divcoord},
         we immediately get that
         \beq
         \label{dnb}
          \| \chi \partial_{n}u \cdot n \|_{m-1} \leq C_{m} \|u \|_{m}.
          \eeq
        It thus remains to estimate $\| \chi \Pi (\partial_{n}u ) \|_{m-1}.$ Let us set
        \beq
        \label{etab1def}
         \eta =  \chi\Pi \Big( \big(\nabla u +  \nabla u^t \big) n \Big) + 2 \alpha  \chi \Pi u= 
       \chi\Pi \Big( Su\,  n \Big) + 2 \alpha  \chi \Pi u   .\eeq
        In view of  the Navier condition \eqref{N}, we obviously have
         that $\eta$ satisfies an homogeneous Dirichlet boundary condition on the boundary:
        \beq
        \label{dirb}
        \eta_{/\partial \Omega}=0.
        \eeq
        Moreover, since an alternative way to write $\eta$ in the vicinity of the boundary  is
        \beq
        \label{etab2}
         \eta=  \chi  \Pi \partial_{n} u+ \chi\Pi \Big( \nabla(u \cdot n) - Dn \cdot u - \, u \times (\nabla \times n)
          + 2 \alpha  u\Big),
          \eeq
          we immediately get that
      $$
       \| \chi\,  \Pi \partial_{n} u \|_{m-1} \leq C_{m+1}\big( \|\eta  \|_{m-1} +  \| u\|_{m} + \|\partial_{n} u \cdot n  \|_{m-1}
        \big).$$
        and  hence thanks to \eqref{dnb} that
       \beq
       \label{etabu}
        \| \chi \Pi \partial_{n} u \|_{m-1} \leq C_{m+1}\big( \|\eta  \|_{m-1} +  \| u\|_{m} \big).
        \eeq
        As before, it  is thus equivalent to estimate $ \|\Pi \partial_{n} u \|_{m-1} $ or $ \|\eta  \|_{m-1}$.
        Note that we  have taken a slightly different definition for $\eta$  in comparison with
         the half space case. The reason is that  it is better to compute the evolution equation for
          $\eta $ with the expression   \eqref{etab1def}  than with the expression \eqref{etab2}
           or with the  expression involving the vorticity.  Indeed,  these last two
            forms require a boundary with more regularity. The price to pay will be that   since we do not
            use the vorticity,  the pressure will again appear in our estimates.
           
           We shall establish the following conormal estimates for $\eta$:
        \begin{prop}
        \label{normb}
        For every $m \geq 1$, we have that
        \begin{eqnarray}
        \label{estnormb}
 & &  \| \eta (t ) \|_{m-1}^2 + \eps \int_{0}^t \| \nabla \eta \|_{m-1}^2    
   \leq  C_{m+2} \big( \|u(0) \|_{m}^2 + \| \nabla u(0) \|_{m-1}^2\big)  \\
    & & \nonumber +  C_{m+2} \int_{0}^t
   \Big(  \big( \|\nabla^2 p_{1} \|_{m-1} + \| \nabla p \|_{m-1}\big) \| \eta \|_{m} +  \eps^{-1} \|\nabla p_{2} \|_{m-1}^2   \\
   & & \quad \quad \quad \quad   \nonumber + \big( 1 + \|u \|_{2, \infty} + \|\nabla u \|_{1, \infty} \big) \big(  \| \eta \|_{m-1}^2 + \|u \|_{m}^2 + \| \nabla u \|_{m-1}^2  \big) \Big)
        \end{eqnarray}
        
        \end{prop}
      Note that by combining Proposition \ref{conormb}, Proposition \ref{normb} and  \eqref{dnb}, \eqref{etabu},
       we immediately obtain the global estimate 
    \begin{eqnarray}
        \label{estnormb+}
 & &  \|u(t) \|_{m}^2 +  \| \nabla u (t ) \|_{m-1}^2 + \eps \int_{0}^t \| \nabla \eta \|_{m-1}^2    
   \leq  C_{m+2} \big( \|u(0) \|_{m}^2 + \| \nabla u(0) \|_{m-1}^2\big)  \\
    & & \nonumber +  C_{m+2} \int_{0}^t
   \Big(   \|\nabla^2 p_{1} \|_{m-1} \big( \| u \|_{m}+
    \| \nabla u \|_{m-1} \big) +  \eps^{-1} \|\nabla p_{2} \|_{m-1}^2   \\
   & & \quad \quad \quad \quad   \nonumber + \big( 1 + \|u \|_{2, \infty} + \|\nabla u \|_{1, \infty} \big) \big(   \|u \|_{m}^2 + \| \nabla u \|_{m-1}^2  \big) \Big)
        \end{eqnarray}
        for $m \geq 2$.
        
        \subsection*{Proof of Proposition \ref{normb}}

        Note that  $M=\nabla u$ solves the equation
        $$ \partial_{t} M + u \cdot \nabla M - \eps \Delta M= - M^2 - \nabla^2 p$$
        where $\nabla^2 p$ denotes  the Hessian matrix of the pressure. 
        Consequently, we get that $\eta$ solves the equation
        \beq
        \label{eqetab}
        \partial_{t} \eta + u \cdot \nabla \eta - \eps \Delta \eta=   F - \chi \Pi\big( \nabla^2 p\, n\big)
        \eeq
        where the source term $F$ can be decomposed into
        \beq
        \label{eqFb}
        F= F^b+ F^\chi+ F^\kappa
        \eeq
        where : 
        \begin{eqnarray}
        \label{Fb}
       & &  F^b= -\chi \Pi \big((\nabla u)^2+ (\nabla u^t)^2\big)n - 2 \alpha \chi \Pi  \nabla p, \\
        \label{Fchi}
      & &   F^\chi=  -  \eps \Delta \chi \Big( \, \Pi Su\,  n + 2 \alpha \Pi u\Big)
       \nonumber    -  2 \eps \nabla \chi \cdot \nabla \Big(   \, \Pi  Su\,  n + 2 \alpha \Pi u\Big) \\
       \nonumber    & &  \hspace{1cm} + (u \cdot \nabla \chi) \Pi \Big(  \big(  Su\, n + 2 \alpha  u\Big),  \\
        \label{Fkappa}  & & F^\kappa=\chi \big( u \cdot \nabla \Pi\big)\Big(   Su \, n + 2 \alpha u \Big)
            + \chi  \Pi  \big( Su\, \big( u \cdot \nabla n \big)  \big) \\
 & &  \hspace{1cm}    \nonumber         - \eps  \chi\big( \Delta \Pi \big)\Big( Su \, n + 2 \alpha u \big)  - 2 \eps \chi \nabla \Pi \cdot \nabla
             \big( Su \, n + 2 \alpha u \big) \\
 & &   \hspace{1cm}   \nonumber          -  \eps \chi \Pi \Big(  Su \, \Delta n  + 2  \nabla  Su \cdot \nabla n  \Big). 
                     \end{eqnarray}

      Let us start with the proof of the $L^2$ energy estimate i.e.  the case $m=1$ in Proposition
       \ref{normb}. By multiplying \eqref{eqetab} by $\eta$, we immediately get that
      \beq
      \label{L2etab} {d \over dt} {1 \over 2} \| \eta \|^2 + \eps \| \nabla \eta \|^2 = \int_{\Omega} F \cdot \eta
      - \int_{\Omega} \chi \Pi\big( \nabla^2p\, n \big)  \cdot \eta.\eeq
      To estimate the right handside, we note that
      \beq
      \label{Fbm-1}
      \|F^b \|_{m-1} \leq C_{m}\Big(  \|u \|_{W^{1, \infty}} \|\nabla u \|_{m-1}  + \|\nabla p \|_{m-1} \Big)
      \eeq
      and also that
      \beq
      \label{Fchim-1} \|F^\chi\|_{m-1} \leq C_{m+1} \Big(  \eps \|\nabla u \|_{m}+
   \big( 1+    \|u \|_{W^{1, \infty}}\big) \|  u \|_{m} \Big).
       \eeq
       Note that we have used that  since all the terms in $F^\chi$ are supported away from the boundary, 
        we can control all the derivatives by the $\| \cdot \|_{m}$ norms. Finally, we also have that
        \beq
        \label{Fkappam-1} \| F^{\kappa} \|_{m-1} \leq C_{m+2} \Big( \eps \|u \|_{m} + \eps  \| \nabla u \|_{m-1}
         + \eps \| \chi \nabla^2 u \|_{m-1}
         + \|u \|_{W^{1, \infty } }( \|u \|_{m-1} + \| \nabla u \|_{m-1}) \Big).\eeq
         To estimate the last term in the right hand side of  \eqref{eqetab},  we split the pressure to get
       $$ \Big|  \int_{\Omega} \chi \Pi\big( \nabla^2p\, n \big)  \cdot \eta \Big|
        \leq \| \nabla^2  p_{1}\|\, \| \eta \|  +  \Big|  \int_{\Omega} \chi \Pi\big( \nabla^2p_{2}\, n \big)  \cdot \eta \Big|.$$
        Since $\eta$ vanishes on the boundary, we can integrate by parts the last term to obtain
     $$  \Big|  \int_{\Omega} \chi \Pi\big( \nabla^2p_{2}\, n \big)  \cdot \eta \Big|
      \leq C_{2} \| \nabla p_{2}\| \big( \| \nabla \eta \| + \| \eta \| \big).$$ 
      Consequently, by plugging these estimates into \eqref{L2etab}, we immediately get that
     \begin{eqnarray}
   \label{etabL21}  && 
      {d \over dt} {1 \over 2} \| \eta \|^2 + \eps \| \nabla \eta \|^2 \\
       & & 
\nonumber   \leq    C_{3}\Big( \big( \eps \| \nabla  u \|_{1} + \eps  \|\chi\nabla^2 u \| \big) \, \|\eta \| +   \|\nabla p_{2} \|  \big( \|\nabla \eta  \|+ \| \eta \|\big)  \\
 \nonumber  & & +\big( \|\nabla ^2 p_{1} \| + \| \nabla p_{1} \| \big) 
     \| \eta\| 
      +( 1 + \|u \|_{W^{1, \infty}}) \big( \|u \|_{1}+ \| \nabla u \| \big) \Big).\end{eqnarray}  
     To conclude, we only need to estimate $ \eps \| \chi\nabla^2 u \|$.
      Note that we have that
      $$   \eps \| \chi\nabla^2 u \| \lesssim \eps  \| \chi \nabla \partial_{n} u \| + \eps  C_{2} \| \nabla u \|_{1}$$
       and hence, by using \eqref{dnb} and \eqref{etab2} that
       $$ \|\chi \nabla \partial_{n} u \| \leq C_{3} \Big( \| \nabla \eta \| +  \|\nabla u \|_{1} +  \|u\|_{1}\Big).$$
       Consequently, by using \eqref{etabL21} and the Young inequality, we finally get that
   \begin{eqnarray*}
 & &     {d \over dt} {1 \over 2} \| \eta \|^2 +{ \eps \over 2} \| \nabla \eta \|^2    \\  
& &       \leq    C_{3}\Big(   \eps \| \nabla  u \|_{1}  \, \|\eta \| +   \big(\|\nabla p \|  + \|\nabla ^2 p_{1} \|\big) \| \eta \| + 
 \eps^{-1} \| \nabla p_{2} \|^2 
     +( 1 + \|u \|_{W^{1, \infty}}) \big( \|u \|_{1}+ \| \nabla u \| \big) \Big).
\end{eqnarray*}          
   Since  $\eps \| \nabla u \|_{1}$  is already estimated in Proposition \ref{conormb}, this yields \eqref{estnormb}  for $m=1$.
  
  To prove the general case, let us assume that \eqref{estnormb} is proven for $k \leq m-2.$ 
   We get from  \eqref{eqetab}  for $|\alpha |= m-1$ that 
   $$ \partial_{t} Z^\alpha \eta + u \cdot \nabla Z^\alpha \eta - Z^\alpha \Delta \eta = Z^\alpha F- Z^\alpha \big(
    \chi \Pi (\nabla^2p \, n ) \big) + 
    \mathcal{C}$$
    where
    $$\mathcal{C}= -  [Z^\alpha, u \cdot \nabla ] \eta.$$ 
     A standard energy estimate yields
    \beq
    \label{etabm-11} {d \over dt } {1 \over 2} \|Z^\alpha  \eta \|^2 \leq 
   \eps    \int_{\Omega} Z^\alpha \Delta \eta \cdot Z^\alpha \eta 
      + \big(\|F \|_{m-1}+ \| \mathcal{C}\|\big) \, \|\eta\|_{m-1} -  \int_{\Omega}
     Z^\alpha \big(
    \chi \Pi (\nabla^2p \, n ) \big) \cdot Z^\alpha \eta 
      .\eeq
        To estimate
       the first term in the right hand side,  we need to estimate
       $$ I_{k} = \int_{\Omega} Z^\alpha \partial_{kk} \eta \cdot Z^\alpha \eta, \quad k=1,\, 2, \, 3.$$
       Towards this, we write
       \begin{eqnarray*}
        I_{k} & = &   \int_{\Omega}  \partial_{k} Z^\alpha \partial_{k} \eta \cdot Z^\alpha \eta+
        \int_{\Omega}  [Z^\alpha, \partial_{k}] \partial_{k} \eta \cdot Z^\alpha \eta \\
          & = &
          - \int _{\Omega}  | \partial_{k} Z^\alpha \eta|^2 
           - \int_{\Omega} [Z^\alpha , \partial_{k}] \eta \cdot \partial_{k} Z^\alpha \eta
          +   \int_{\Omega}  [Z^\alpha, \partial_{k}] \partial_{k} \eta \cdot Z^\alpha \eta.
          \end{eqnarray*}      
      Note that there is no boundary term in the integration by parts since $Z^\alpha \eta$
       vanishes on the boundary. To estimate the  last two  terms above, 
       we need to use the structure of the commutator $[Z^\alpha, \partial_{k}].$
        By using the expansion
        $$\partial_{k}= \beta^1 \partial_{y^1} + \beta^2 \partial_{y^2} + 
         \beta^3 \partial_{y^3},$$
     in the local basis, we get  an expansion under the form
     $$ [Z^\alpha, \partial_{k}]f
     =  \sum_{\gamma,  |\gamma| \leq | \alpha |-1} c_{\gamma} \partial_{z} Z^\gamma f
      + \sum_{ \beta, \, |\beta| \leq |\alpha|} c_{\beta} Z^\beta  f $$
      where the   $\mathcal{C}^l$ norm of the coefficients is  bounded
       by $C_{l+m}$. This yields the estimates 
 \begin{eqnarray*}  
 & & \Big|   \int_{\Omega} [Z^\alpha , \partial_{k}] \eta \cdot \partial_{k} Z^\alpha \eta \Big| \\
 & & \leq   C_{m} \| \nabla \eta \|_{m-2} \, \|\nabla Z^{m-1} \eta\|
 \end{eqnarray*}
 and
  \begin{eqnarray*}  
 & & \Big|   \int_{\Omega} [Z^\alpha , \partial_{k}] \partial_{k} \eta \cdot  Z^\alpha \eta \Big| \\
 & & \leq  \sum_{|\gamma|\leq m-2} \Big| \int_{\Omega} c_{\gamma} \partial_{z} Z^\gamma \partial_{k} \eta \cdot
  Z^\alpha \eta  \Big| + C_{m} \| \nabla \eta \|_{m-1} \, \| \eta \|_{m-1}. 
 \end{eqnarray*}
Since $Z^\alpha\eta$ vanishes on the boundary, this yields thanks to an integration by parts 
$$
\Big|   \int_{\Omega} [Z^\alpha , \partial_{k}] \partial_{k} \eta \cdot  Z^\alpha \eta \Big| 
 \leq C_{m+1}  \|\nabla \eta\|_{m-1}\big( \| \nabla \eta \|_{m-2}  + \| \eta \|_{m-1}\big).$$
   Consequently, we get from \eqref{etabm-11} by summing over $\alpha$ and a new use of the Young inequality that
\begin{eqnarray}
 & &\label{etabm-12}   {d \over dt } {1 \over 2} \|\eta \|_{m-1}^2
 + {\eps \over 2} \| \nabla Z^{m-1} \eta \|^2 \\ 
& & \nonumber   \leq C_{m+1} \Big( \eps \| \nabla \eta \|_{m-2}^2+ \| \eta \|_{m-1}^2 +  \big(\|F \|_{m-1}+ \| \mathcal{C}\|\big) \, \|\eta\|_{m-1}\Big) -  \int_{\Omega}
     Z^\alpha \big(
    \chi \Pi (\nabla^2p \, n ) \big) \cdot Z^\alpha \eta .
  \end{eqnarray}
To estimate the right hand side, we first notice that to control  the term involving $F$,
      we can use  \eqref{Fbm-1}, \eqref{Fchim-1} and \eqref{Fkappam-1}. This yields
   \begin{eqnarray}
   \label{Fbm-11}
   \| F \|_{m-1} & \leq &  C_{m+2} \Big( 
    \big(\eps \| \nabla u \|_{m} +  \eps \| \chi \nabla^2 u \|_{m-1} + \| \nabla p \|_{m-1} 
     \big) \|\eta \|_{m} \\
      & & \nonumber  +  ( 1 + \|u \|_{W^{1, \infty}} ) \big( \|u \|_{m}
     + \|\nabla u \|_{m-1}\big) \Big)
   \end{eqnarray}
   It remains to estimate  $\eps \| \chi \nabla^2 u \|_{m-1}$. We can first use that
   $$  \eps \| \chi \nabla^2 u \|_{m-1} \leq \eps \| \chi \nabla \partial_{n} u \|_{m-1}
    + \eps C_{m+1} \big(  \| \nabla u \|_{m} +   \|u \|_{m}\big).$$
Next, thanks to \eqref{etab2} and \eqref{dnb}, we also get that
$$ \eps \| \chi \nabla \partial_{n} u \|_{m-1}
 \leq C_{m+2} \Big(   \eps   \| \nabla u \|_{m} + \|u \|_{m} + \| \nabla \eta \|_{m-1} \Big)$$
  and hence we obtain the estimate
  \begin{eqnarray}
   \label{Fbm-12}
   \| F \|_{m-1} & \leq &  C_{m+2} \Big( 
    \big(\eps \| \nabla u \|_{m} +  \eps \| \nabla \eta  \|_{m-1} + \| \nabla p \|_{m-1} 
     \big) \|\eta \|_{m} \\
      & & \nonumber  +  ( 1 + \|u \|_{W^{1, \infty}} ) \big( \|u \|_{m}
     + \|\nabla u \|_{m-1}\big) \Big).
   \end{eqnarray} 
   
   In view of \eqref{etabm-12}, it remains to estimate $\| \mathcal{C} \|$.
     Note that  by using the local coordinates, we can expand:
     $$ u \cdot  \nabla \eta = u_{1} \partial_{y^1} \eta + u_{2} \partial_{y^2}  \eta+ u \cdot \n\, \partial_{z} \eta. $$
      Consequently,  the estimate  \eqref{C1n} also holds for this term, we thus get that
      \beq
      \label{Cbeta}  \| \mathcal{C} \| \leq C_{m} \Big( \|u \|_{2, \infty} + \|u \|_{W^{1, \infty}}
       + \| Z \eta \|_{L^\infty} \Big)\big( \|\eta \|_{m-1} + \|u \|_{m} \big).\eeq
       
     Finally, it remains to estimate the last term involving the pressure in the right hand side of
       \eqref{etabm-12}. As before, we use the splitting $p=p_{1}+ p_{2}$ and we integrate
        by parts the term involving $p_{2}$. This yields
    \beq
    \label{palphafin} \Big| \int_{\Omega}
     Z^\alpha \big(
    \chi \Pi (\nabla^2p \, n ) \big) \cdot Z^\alpha \eta \Big|
     \leq  C_{m+ 2 } \big(  \| \nabla^2 p_{1}\|_{m-1} \, \| \eta \|_{m} +  \| \nabla p_{2} \|_{m-1} \big( \| \nabla Z^m \eta \|
      + \| \eta \|_{m}\big)\Big).\eeq
      By combining \eqref{etabm-12}, \eqref{Fbm-12},  \eqref{Cbeta}, \eqref{palphafin}   and 
       by using the induction assumption and the Young inequality, we get the
      result.
      
   \subsection{Pressure estimates}
   \begin{prop}
   \label{pressureb}
  For $m \geq 2$,  we have the following estimate for the pressure:
   \begin{eqnarray}
   \label{estpressureb1}
 & &   \| \nabla p_{1} \|_{m-1}  + \| \nabla^2 p_{1} \|_{m-1} 
  \leq C_{m+2}
       \big( 1 + \|u \|_{W^{1, \infty}} \big) \big( \|u \|_{m} + \|\nabla u\|_{m-1} \big), \\
       & & \label{estpressureb2}
        \| \nabla p_{2} \|_{m-1} \leq C_{m+2} \, \eps\, \big(  \| \nabla u \|_{m-1} + \|u \|_{m} \big).
      \end{eqnarray}
   \end{prop}
     Note  that  thanks to \eqref{estpressureb2}, we  have that
   $$ \eps^{-1} \| \nabla p_{2} \|_{m-1}^2 \leq C_{m+2} \big( \|u \|_{m}^2 + \| \nabla u \|_{m-1}^2 \big).$$
   Consequently, by combining \eqref{estnormb+} and Proposition \ref{pressureb}, we get that
   \begin{eqnarray}
        \label{estnormbglob1}
& &   \|u(t) \|_{m}^2 +  \| \nabla u (t ) \|_{m-1}^2 + \eps \int_{0}^t \| \nabla^2  u  \|_{m-1}^2   \\
\nonumber   & &  \leq C_{m+2} \big(   \|u_{0}\|_{m}^2 + \| \nabla u_{0} \|_{m}^2 \big)  
 + C_{m+2}
   \int_{0}^t  \Big( \big( 1 + \|u \|_{2, \infty}
 + \|\nabla u \|_{1, \infty} \big) \big(   \|u \|_{m}^2 + \| \nabla u \|_{m-1}^2  \big) \Big)
        \end{eqnarray}

   \subsubsection*{Proof}
   We recall that we have  $p=p_{1}+ p_{2}$ where
   \beq
   \label{p1b}
   \Delta p_{1}=  - \nabla \cdot(u \cdot \nabla u)= -  \nabla u \cdot \nabla u, \quad x \in \Omega, \quad
    \partial_{n}p_{1}=  - ( u\cdot \nabla u ) \cdot n, \quad x \in \partial \Omega
    \eeq
     and
     \beq
     \label{p2b}
     \Delta p_{2}= 0, \quad x \in \Omega, \quad  \partial_{n}p_{2} = \eps  \Delta u \cdot n, \quad  x \in \partial \Omega.
     \eeq
     From standard  elliptic regularity results with  Neumann boundary conditions, we get that
   $$    \| \nabla p_{1} \|_{m-1} + \| \nabla^2 p_{1} \|_{m-1} \leq C_{m+1} \Big( \| \nabla u  \cdot \nabla u \|_{m-1}+
    \|u \cdot \nabla u \| + |\big( u \cdot \nabla u \big) \cdot n |_{H^{m- {1 \over 2}}(\partial \Omega)} \Big).$$
    Since $u \cdot n=0$ on the boundary, we note that
    $$  ( u \cdot \nabla u \big) \cdot n= - \big( u \cdot \nabla n \big) \cdot u, \quad x \in \partial \Omega$$
     and consequently, thanks to the trace Theorem, we obtain that
     $$   |\big( u \cdot \nabla u \big) \cdot n |_{H^{m- {1 \over 2}}(\partial \Omega)}
      \leq C_{m+2} \big( \| \nabla \big(u \otimes u\big) \|_{m-1}+   \| u \otimes u \|_{m-1} \big).$$
       Thanks to a new use of Lemma \ref{gag}, this yields
    $$   \| \nabla p_{1} \|_{m-1} + \| \nabla^2 p_{1} \|_{m-1} \leq C_{m+2} \Big(
     \big( 1 + \|u \|_{W^{1, \infty}} \big) \big( \|u \|_{m} + \|\nabla u\|_{m-1} \big) \Big).$$
     It remains to estimate $p_{2}$. By using again  the elliptic regularity for the Neumann problem, 
     we get that for $m \geq 2$, 
     \beq
     \label{nablap21}   \| \nabla p_{2} \|_{m-1}  \leq  \eps\,  C_{m} | \Delta u \cdot n |_{H^{m- {3  \over 2}}(\partial \Omega)}.\eeq
     To estimate the right hand side,  we shall again use the Navier boundary condition \eqref{N}.
      Since
      $$ 2 \Delta u \cdot n = \nabla \cdot \big( Su \, n) - \sum_{j} \big(Su\, \partial_{j} n \big)_{j},$$
      we first get that
     $$  | \Delta u \cdot n |_{H^{m- {3  \over 2}}(\partial \Omega)}
      \lesssim  | \nabla \cdot \big( Su \, n) |_{H^{m- {3  \over 2}}(\partial \Omega)}
       +  C_{m+1} | \nabla u |_{H^{m- {3  \over 2}}(\partial \Omega)}$$
        and hence thanks to  \eqref{divcoord} and \eqref{N2} that
      $$  | \Delta u \cdot n |_{H^{m- {3  \over 2}}(\partial \Omega)}
            \lesssim  | \nabla \cdot \big( Su \, n) |_{H^{m- {3  \over 2}}(\partial \Omega)} + C_{m+1} |u |_{H^{m- {1 \over 2}}(\partial \Omega)}.$$
            To estimate the first term, we can use  the expression \eqref{divcoord} to get
        $$   | \nabla \cdot \big( Su \, n) |_{H^{m- {3  \over 2}}(\partial \Omega)}
         \lesssim   | \partial_{n} \big( Su \, n) \cdot n  |_{H^{m- {3  \over 2}}(\partial \Omega)}
        + C_{m+1} \Big(  |\Pi \big( Su \, n \big) |_{H^{m - {1 \over 2 }}(\partial \Omega)} +
        |\nabla u |_{H^{m - {3 \over 2 }} (\partial \Omega)}\Big)$$
       and hence by using again \eqref{divcoord},  \eqref{N2} and \eqref{N}, we obtain that
       $$   | \nabla \cdot \big( Su \, n) |_{H^{m- {3  \over 2}}(\partial \Omega)}
         \lesssim   | \partial_{n} \big( Su \, n) \cdot n  |_{H^{m- {3  \over 2}}(\partial \Omega)}
        + C_{m+1}
        | u |_{H^{m - {1 \over 2 }} (\partial \Omega)}.$$
        The first term above  in the right hand side  can  be estimated by 
       \begin{eqnarray*}    | \partial_{n} \big( Su \, n) \cdot n  |_{H^{m- {3  \over 2}}(\partial \Omega)}
       &  \lesssim &  | \partial_{n} \big( \partial_{n} u \cdot n \big) |_{H^{m- {3  \over 2}}(\partial \Omega)}
        +  C_{m+1} | \nabla u |_{H^{m- {3  \over 2}}(\partial \Omega)} \\
        & \lesssim   &  | \partial_{n} \big( \partial_{n} u \cdot n \big) |_{H^{m- {3  \over 2}}(\partial \Omega)}
          +   C_{m+1} |u|_{H^{m- {1 \over 2}}(\partial \Omega)}.
          \end{eqnarray*}
        Finally,   taking the normal derivative of \eqref{divcoord}, we get that
      \begin{eqnarray*}
       | \partial_{n} \big( \partial_{n} u \cdot n \big) |_{H^{m- {3  \over 2}}(\partial \Omega)}
      &  \lesssim &  | \Pi \partial_{n}u |_{H^{m- {1 \over 2 }}(\partial \Omega)}
        +  C_{m+1} | \nabla u |_{H^{m - {3 \over 2}}(\partial \Omega)} \\
        & \lesssim&   C_{m+2}\,  |u |_{H^{m- {1 \over 2}} (\partial \Omega)}
        \end{eqnarray*}
        where the last line comes from a new use of \eqref{N2}. Note that this is the estimate of this
        term which requires the more regularity of the boundary.
        
        Consequently, we have  proven that
      $$  | \Delta u \cdot n |_{H^{m- {3  \over 2}}(\partial \Omega)}
 \leq C_{m+2}  |u |_{H^{m- {1 \over 2}} (\partial \Omega)}$$
 and hence by using \eqref{nablap21} and the trace Theorem, we get that
 $$ \| \nabla p_{2} \|_{m-1} \leq C_{m+2} \eps  \big( \|u \|_{m} + \| \nabla u \|_{m-1} \big).$$
 
 This ends the proof of Proposition \ref{pressureb}.

    \subsection{$L^\infty$ estimates}
    \label{supb}
    In order to close the estimates, we need  an estimate of the $L^\infty$ norms in the right hand side.
     As before, let us set
     $$ N_{m}(t)=  \|u(t) \|_{m}^2 + \| \nabla u(t) \|_{m-1}^2+  \| \nabla u \|_{1, \infty}^2.$$
     \begin{prop}
     \label{Linftyb1} For $m_{0}>1$, we have
     \begin{eqnarray} 
     & & \|u\|_{2, \infty} \leq  C_{m}\big( \|u\|_{m} + \| \nabla u \|_{m-1}\big), \quad m \geq m_{0}+ 3,  \\
     & & \| u \|_{W^{1, \infty}} \leq C_{m} \big( \|u \|_{m} + \| \nabla u \|_{m-1} \big), \quad m \geq m_{0}+ 2. 
      \end{eqnarray}
      \end{prop}
      \subsubsection*{Proof}
      It suffices to use local coordinates and Proposition \ref{Linfty}.
      
      In view of Proposition,  we still need to estimate $\|\nabla u \|_{1, \infty}.$
      
      \begin{prop}
      \label{nablausupb}
      For $m>6$, we have the estimate
      $$ \| \nabla u(t) \|_{1, \infty}^2 \leq
       C_{m+2} \Big( N_{m}(0) + (1+ t + \eps^3 t^2 \big) \int_{0}^t \big(N_{m}(s) + N_{m}(s)^2 \big) \, ds.$$
      \end{prop}
      \subsubsection*{Proof}
      Away from the boundary, we clearly have by  the classical isotropic Sobolev embedding that
      \beq
      \label{intbinfty} \|\chi \nabla u \|_{1, \infty} \lesssim \|u\|_{m}, \quad m \geq  4.\eeq
      Consequently, by using  a partition of unity subordinated to the covering
        \eqref{covomega}  we only have to estimate $\|\chi_{i} \nabla u \|_{L^\infty}$, $i >0$.
        For  notational convenience, we shall denote $\chi_{i}$ by $\chi$.
        Towards this, we want to proceed as in the proof of Proposition \ref{etainfty}.
          An important step in this proof was  to use Lemma \ref{FP}.
          It is thus crucial to choose a system of coordinates in which the Laplacian has
           a convenient   form. In this section, we shall use a local parametrization in the vicinity 
           of the boundary given by  a normal geodesic system: 
         $$ \Psi^n(y,z)= \left( \begin{array}{ll} y \\ \psi(y) \end{array} \right)  - z\,n(y)$$
         where
        $$ n( y)= { 1 \over \big(1 +| \nabla \psi(y) |^2\big)^{1 \over 2} }
            \left(  \begin{array}{ll}  \partial_{1} \psi(y) \\
                   \partial_{2} \psi(y) \\ - 1  \end{array} \right)$$
                    is the unit outward normal.
       We have not used this coordinate system to estimate the conormal derivatives  because it requires more regularity on
         the boundary.  Nevertheless, it  does not  yield  any restriction on the regularity  of the  boundary here, 
          since we need to estimate a lower number of derivatives.          
        As before, we can extend $n$ and $\Pi$ in the  interior by setting
        $$ n(\Psi^n(y,z))= n(y), \quad \Pi(\Psi^n(y,z))= \Pi(y).$$
        Note that $n(y)$ and $\Pi(y)$ have different definitions from the ones used before. 
         The interest of this parametrization is that  in the associated  local basis  of $\mathbb{R}^3$
         $(\partial_{y^1}, \partial_{y^2}, \partial_{z})$, we have  $\partial_{z}= \partial_{n}$ and
         $$ \Big(\partial_{y^i}\Big)_{/\Psi^n(y,z)} \cdot \Big(\partial_{z} \Big)_{/\Psi^n(y,z)}= 0.$$
        The  scalar product on $\mathbb{R}^3$ thus  induces
          in this coordinate system  the  Riemannian metric $g$   under the form
          \beq
          \label{gform}
           g(y,z)= \left( \begin{array}{cc} \tilde{g}(y,z) &  0 \\ 0 &  1 \end{array} \right).
           \eeq
           Consequently, the Laplacian in this coordinate system reads:
           \beq
           \label{laplacian}
           \Delta f= \partial_{zz} f + {1 \over 2} \partial_{z}\big( \ln |g| \big) \partial_{z} f + \Delta_{\tilde{g}}f
           \eeq
           where $|g|$ denotes the determinant of the matrix $g$ and $\Delta_{\tilde{g}}$ which is defined by 
           $$ \Delta_{\tilde{g}} f=  { 1 \over |\tilde{g}|^{1 \over 2} }  \sum_{1 \leq i, \, j \leq 2} \partial_{y^i}\big(
            \tilde{g}^{ij} |\tilde{g}|^{1 \over 2} \partial_{y^j} f \big)$$
            involves only tangential derivatives.

                  Next, we can  observe that thanks to  \eqref{divcoord} (in  the coordinate system that we have just defined)  and Proposition \ref{Linftyb1} we have that
      \beq
      \label{nablabetasup1}
     \|\chi  \nabla u \|_{1, \infty }  \leq C_{3}\big( \| \chi \Pi \partial_{n} u \|_{1, \infty} +
     \| u\|_{2, \infty} \big)  \leq C_{3} \big( \| \chi \Pi \partial_{n} u \|_{1, \infty} 
      +  \|u\|_{m} + \|\nabla u \|_{m-1}\big). 
      \eeq
      Consequently, we need to estimate  $\| \chi \Pi \partial_{n} u \|_{1, \infty}$.
      To estimate this quantity,  it is useful to introduce the vorticity 
      $$\omega = \nabla \times u.$$
      Indeed, by definition, we have
      \beq
      \label{omegaint} \Pi  \big( \omega \times n\big)= {1 \over 2} \Pi \big( \nabla u - \nabla u^t)n=
       {1 \over 2} \Pi\Big( \partial_{n} u  - \nabla (u \cdot n ) +  u \cdot \nabla n +  u \times \big( \nabla \times n)\Big).\eeq
       Consequently, we find that
       $$   \| \chi \Pi \partial_{n} u \|_{1, \infty} \leq C_{3} \Big(  \| \chi \Pi  ( \omega \times n)  \|_{1, \infty}
        + \|u \|_{2, \infty} \Big)$$
        and hence by a new use of Proposition \ref{Linftyb1} , we get that 
   \beq
   \label{chibdn1}   \| \chi \Pi \partial_{n} u \|_{1, \infty}      \leq  C_{3} \Big(  \| \chi  \Pi \big(\omega \times n \big)   \|_{1, \infty}
        + \|u \|_{m} + \|\nabla u \|_{m-1} \Big).\eeq
        In other words, we only need to estimate $ \| \chi \Pi \big( \omega \times n \big)\|_{1, \infty}$ in order to conclude.
        Note that $\omega$ solves the vorticity equation
        \beq
        \label{vortb}
        \partial_{t } \omega + u \cdot \nabla \omega - \eps \Delta \omega =  \omega \cdot \nabla u = F^\omega.\eeq
        Consequently, by setting   in the support of $\chi$ 
        $$ \tilde{\omega}(y,z)= \omega(\Psi^n(y,z)), \quad \tilde{u}(y,z)= u(\Psi^n(y,z)), $$
        we get that
        \beq
        \label{tildeomegab}
        \partial_{t} \tilde{\omega}+ \tilde{u}^1 \partial_{y^1} \tilde{\omega} + \tilde{u}^2 \partial_{y^2}  \tilde{\omega}+
         \tilde{u} \cdot n\, \partial_{z} \tilde{\omega} = 
         \eps\big( \partial_{zz} \tilde{\omega}  + {1 \over 2} \partial_{z}\big( \ln |g| \big) \partial_{z} \tilde{\omega} + \Delta_{\tilde{g}}\tilde{\omega}\big) + F^\omega \eeq
         and
         \beq
         \label{tildeub}
        \partial_{t} \tilde{u}+ \tilde{u}^1 \partial_{y^1} \tilde{u} + \tilde{u}^2 \partial_{y^2}  \tilde{u}+
         \tilde{u} \cdot n\, \partial_{z} \tilde{u} = 
         \eps \big(\partial_{zz} \tilde{u}  + {1 \over 2} \partial_{z}\big( \ln |g| \big) \partial_{z} \tilde{u} + \Delta_{\tilde{g}}\tilde{u}\big)  -\big(\nabla p\big) \circ\Psi^n.
         \eeq
         Note that we use the same convention as before for a vector $u$, $u^i$ denotes the components
          of $u$ in the local basis $(\partial_{y^1}, \, \partial_{y^2}, \, \partial_{z})$
           whereas $u_{i}$ denotes it components in the canonical basis of $\mathbb{R}^3$.
           The vectorial equations \eqref{tildeomegab} and \eqref{tildeub} have to be understood
             components by components in the standard basis of  $\mathbb{R}^3$.
             
         By using \eqref{omegaint} on the boundary and  the Navier boundary condition \eqref{N2}, we get
         that  for $z=0$
     $$        \Pi (\tilde{\omega} \times n)=   \Pi\big( \tilde{u} \cdot \nabla n  - \alpha \tilde{u} \big).$$
      Consequently, we set
      \beq
      \label{chibdn2} \tilde{\eta}(y,z)=   \chi \Pi \Big(\tilde{\omega} \times n -   \tilde{u} \cdot \nabla n + 
      \alpha \tilde{u}\Big).\eeq
      We thus get that
      \beq
      \label{etatildeb}
      \tilde{\eta}(y,0)= 0
      \eeq
      and that $\tilde{\eta}$ solves the equation
     \beq
     \label{etatildeq}
     \partial_{t} \tilde{\eta} +  \tilde{u}^1 \partial_{y^1} \tilde{\eta} + \tilde{u}^2 \partial_{y^2}  \tilde{\eta}+
         \tilde{u} \cdot n\, \partial_{z} \tilde{\eta} =
         \eps \big(\partial_{zz} \tilde{\eta}  + {1 \over 2} \partial_{z}\big( \ln |g| \big) \partial_{z} \tilde{\eta}  \big)  + \chi \Pi F^\omega \times n + F^u
          + F^\chi + F^\kappa     \eeq
          where the source terms are given by 
    \begin{eqnarray}
    \label{Fubtilde}
     & & F^u= \chi \Pi \Big(  \nabla p \cdot \nabla n - \alpha \nabla p  \Big)\circ\Psi^n, \\
  \label{Fchitilde}   & & F^\chi=  \Big(\big( \tilde{u}^1 \partial_{y^1} + \tilde{u}^2 \partial_{y^2} + u \cdot n  \,\partial_{z} \big)
      \chi \Big) \Pi\big( \tilde{\omega} \times n - \tilde u \cdot \nabla n + \alpha \tilde{u}
       \big)\\
  \nonumber     & & \mbox{\hspace{0.5cm}}  - \eps \Big( \partial_{zz} \chi  +  2 \eps \partial_{z} \chi \partial_{z}  +  \eps {1 \over 2} \partial_{z}\big( \ln |g| \big) \partial_{z} \chi   \Big)  \Pi\big( \tilde{\omega} \times n - \tilde u \cdot \nabla n + \alpha \tilde{u}
       \big)\\
     \label{Fkappatilde}  & & F^\kappa = 
     \Big( \big( \tilde{u}^1 \partial_{y^1} + \tilde{u}^2 \partial_{y^2}  \big) \Pi \Big) 
      \tilde{\omega} \times n - \tilde u \cdot \nabla n + \alpha \tilde{u}
       +  \Pi \Big( \tilde{\omega} \big( \tilde{u}^1 \partial_{y^1} + \tilde{u}^2 \partial_{y^2}  \big) n \Big) 
        \\
    \nonumber    & & \mbox{\hspace{0.5cm}}- \Pi \Big(   \big( \big( \tilde{u}^1 \partial_{y^1} + \tilde{u}^2 \partial_{y^2}  \big)   \nabla n \big) u \Big) \\
    \nonumber & &   \mbox{\hspace{0.5cm}} - \eps \Delta_{\tilde{g}} \Big(
      \chi \Pi \Big(\tilde{\omega} \times n -   \tilde{u} \cdot \nabla n + 
      \alpha \tilde{u}\Big) \Big).
                  \end{eqnarray}
      Note  that in computing  the source terms and in particular $F^\kappa$ which contains
      all the  commutators coming from the fact that $\Pi$ and $n$ are not constant, we have
      used that in the coordinate system that we have choosen, $\Pi$ and $n$ do not
      depend on the normal variable.
      By using that $\Delta_{\tilde{g}}$ only involves tangential derivatives and that the derivatives
       of  $\chi$ are  compactly supported away from the boundary, we get the estimates
       \begin{eqnarray*}
     & &  \|F^u \|_{1, \infty} \leq  C_{3} \| \Pi \nabla p \|_{1, \infty}, \\
     & & \| F^\chi \|_{1, \infty }\leq C_{3}\Big(    \|u \|_{1, \infty} \| u \|_{2,\infty} + \eps \|u \|_{3, \infty}\Big), \\
     & & \|F^\kappa \|_{1, \infty} \leq C_{4} \Big(  \|u\|_{1, \infty} \| \nabla u \|_{1, \infty} +
      \eps \big( \|\nabla u\|_{3, \infty} + \|u \|_{3, \infty}\big)\Big).
        \end{eqnarray*}
        Note that the fact that the term $(\nabla p \cdot \nabla )n$ in \eqref{Fubtilde} contains only tangential derivatives
         of the pressure  comes from the block diagonal structure of the metric \eqref{gform} and
          the fact that $n$ does not depend on the normal variable $z$.
          
        Consequently, by using  Proposition \ref{Linftyb1}, we get that
      \beq
      \label{supsourceb}
       \|F\|_{1, \infty} \leq  C_{4} \Big(  \| \Pi \nabla p \|_{1, \infty} +  Q_{m} + \eps 
        \|\nabla u \|_{3, \infty}\Big), \quad m \geq m_{0}+ 4
        \eeq
        where $F= F^u + F^\chi + F^\kappa$.
        
       In order to be able to use Lemma \ref{lemFP0}, we shall perform a last change of unknown
        in order to eliminate the term 
      $ \partial_{z}\big( \ln |g| \big) \partial_{z} \tilde{\eta} $
       in \eqref{etatildeq}. We set
       $$ \tilde{\eta}= {1 \over |g |^{1 \over 4} } \eta= \gamma \eta.$$
       Note that  we have
       \beq
       \label{chibdn3} \| \tilde{\eta} \|_{1, \infty} \lesssim  C_{3}\| \eta \|_{1, \infty}, \quad
         \| \eta \|_{1, \infty} \lesssim  C_{3}\|\tilde{ \eta} \|_{1, \infty}\eeq
        and that moreover, $\eta$ solves the equation
       \begin{eqnarray}
       \label{etabequation}
      & &  \partial_{t} \eta +  \tilde{u}^1 \partial_{y^1} \tilde{\eta} + \tilde{u}^2 \partial_{y^2}  \eta+
         \tilde{u} \cdot n\, \partial_{z} \eta - \eps \partial_{zz} \eta \\
         & &=
        {1 \over \gamma} \big( \chi \Pi F^\omega \times n + F^u
          + F^\chi + F^\kappa + \eps \partial_{zz} \gamma \, \eta
           + {\eps \over 2} \partial_{z} \ln |g| \, \partial_{z} \gamma \, \eta -( \tilde{u} \cdot \nabla
           \gamma) \, \eta  \Big) := \mathcal{S}. 
           \end{eqnarray} 
        Consequently, by using Lemma \ref{lemFP0}, we get that
     \begin{eqnarray*} \| \eta (t) \|_{1, \infty}
         &  \lesssim &  \| \eta_{0} \|_{1, \infty} + 
          \int_{0}^t \Big( \big(  \|\tilde{u} \|_{2, \infty} + \| \partial_{z} \tilde{u} \|_{1, \infty} \big) \big( \| \eta \|_{1, \infty}
           + \| \eta \|_{m_{0}+ 3} \big) + \| \mathcal{S} \|_{1, \infty} \Big) \\
   & \lesssim &      \| \eta_{0} \|_{1, \infty} + 
         C_{3} \int_{0}^t \Big( \big(  \|u \|_{2, \infty} + \| \nabla u  \|_{1, \infty} \big) \big( \| \eta \|_{1, \infty}
           + \| \eta \|_{m_{0}+ 3} \big) + \| \mathcal{S} \|_{1, \infty} \Big).
           \end{eqnarray*}
           Consequently, we can use  \eqref{chibdn1}, \eqref{chibdn2}, \eqref{chibdn3}, \eqref{supsourceb}
           and Proposition  \ref{Linftyb1} to get as in the proof of Proposition \ref{etainfty} that
      \begin{eqnarray*}
       \| \chi \Pi\, \partial_{n} u (t)\|_{1, \infty}^2
      &  \leq&  C_{m+1} \Big( \|u(t) \|_{m}^2 + \| \nabla u(t) \|_{m-1}^2 +  N_{m}(0) + 
         \eps \int_{0}^t \| \nabla^2 u \|_{m-1}^2 \\
         & &  \quad   + (1+ t + \eps^3t ^2 \big)\int_{0}^t\big( N_{m}(s) + N_{m}^2 (s) +
          \| \Pi \, \nabla p \|_{1, \infty}^2 \, \big)ds.
          \end{eqnarray*}
          Since $\Pi  \nabla p$ involves only tangential derivatives, we  get thanks to the anisotropic Sobolev
           embedding that for $m \geq 4$
        $$   \| \Pi \, \nabla p \|_{1, \infty}^2 \leq C_{m}  \| \nabla p \|_{m-1}^2.$$
        Consequently, the proof of Proposition \ref{nablausupb} follows by
         using \eqref{estnormbglob1} and Proposition \ref{pressureb}. 
         
      \subsection{Proof of Theorem \ref{apriorib}}
      It suffices to combine  Proposition \ref{nablausupb} and the estimate \eqref{estnormbglob1}.

        \section{Proof of Theorem \ref{mainb}}
        To prove  that  \eqref{NS}, \eqref{N} is locally well-posed in  the function space $E^m \cap \mbox{Lip}$,
          one can for example smooth the initial data in order to use  a standard 
           well-posedness result and then use   the  priori estimates
            given in Theorem \ref{apriorib} and a compactness argument to prove the local existence
             of  a solution (we shall not give more details since the compactness
              argument is almost the same as the one needed for the proof of Theorem \ref{inviscid}). The uniqueness  of the solution is clear since we work with functions
              with Lipschitz regularity.
           The fact that the  life time of the solution is independent of the viscosity $\eps$  then follows  by
            using again  Theorem \ref{apriorib} and a continuous induction argument.

       \section{Proof of Theorem \ref{inviscid}}
       
       \label{sectioninviscid}
       
       Thanks to Theorem \ref{mainb}, the apriori estimate \eqref{uniftheo} holds on $[0, T]$.
        In particular, for each $t$, $u^\eps(t)$ is bounded in $H^m_{co}$ and
         $\nabla u^\eps(t)$ is bounded in $H^{m-1}_{co}$.  This yields that  for each $t$,  $u^\eps(t)$
          is compact in  $H^{m-1}_{co}$. Next, by using the equation \eqref{NS}, we get that
          $$ \int_{0}^T  \|\partial_{t} u^\eps(t) \|_{m-1}^2
           \leq  \int_{0}^T \big(  \eps^2 \| \nabla ^2 u^\eps \|_{m-1}^2 +  \| \nabla p^\eps \|_{m-1}^2 + 
            \|u^\eps  \cdot \nabla u^\eps \|_{m-1}^2  \big) ds$$
           and hence by using  Lemma \ref{gag} and Proposition \ref{pressureb}, we get
            that
         $$  \int_{0}^T  \|\partial_{t} u^\eps(t) \|_{m-1}^2
           \leq  \int_{0}^T \big( \eps^2   \| \nabla ^2 u^\eps \|_{m-1}^2 + 
           \big( 1 + \|u \|_{m}^2 + \| \nabla u \|_{m-1}^2 + \| \nabla u \|_{L^\infty}^2\big)  \big(
           \|u \|_{m}^2 + \| \nabla u \|_{m-1}^2\big) ds.$$
           Consequently, thanks to the uniform estimate \eqref{uniftheo}, we get that
            $\partial_{t}u^\eps $ is uniformly bounded in $L^2(0, T, H^{m-1}_{co})$.
            
            From the Ascoli Theorem, we thus get that $u^\eps$ is compact in $\mathcal{C}([0, T], H^{m-1}_{co}\big)$.
             In particular, there exists a sequence $\eps_{n}$ and $u\in\mathcal{C}([0, T], H^{m-1}_{co}\big) $
              such that $u^{\eps_{n}}$ converges towards $u$ in  
     $  \mathcal{C}([0, T], H^{m-1}_{co}\big) $.  By using again the uniform bounds
      \eqref{uniftheo}, we get that $u \in \mbox{Lip}$.      Thanks to the anisotropic Sobolev embedding \eqref{sob}, we also have that for $m_{0}>1$
      $$ \sup_{[0, T]} \|u ^{\eps_{n}}(t) - u(t) \|_{L^\infty}^2 \leq  \sup_{[0, T]}\big( \| \nabla \big(u^{\eps_{n}} - u\big) \|_{m_{0}}
        \| u^{\eps_{n}} - u\big \|_{m_{0}} +   \| u^{\eps_{n}} - u\big \|_{m_{0}}^2 \big)
       $$
        and hence  again thanks to the uniform bound \eqref{uniftheo}, we get that $u^{\eps_{n}}$ converges uniformly towards $u$ on $[0, T] \times \Omega$. Moreover, it is easy to  check that $u$ is a weak solution of the Euler equation.
         
         Finally since $u \in L^\infty \big( [0, T ] , L^2 \cap \mbox{Lip}\big)$, $u$ is actually  unique and hence we get that
          the whole family $u^\eps$ converges towards $u$. This ends the proof of Theorem \ref{inviscid}.

  \vspace{0.5cm}      

\subsection*{Acknowledgements}
N. M.  was partially supported by an NSF grant and F. R.  by the ANR project ANR-08-JCJC-0104-01.
This work was carried out during visits to the Universities of Nice and Rennes 1 and the Courant Institute.
The hospitality of these institutions is highly acknowledged.

\end{document}